\documentclass[11pt]{article} 
\usepackage[utf8]{inputenc} 

\usepackage[margin=1in]{geometry} 
\usepackage{graphicx} 

\usepackage{dsfont} 

\usepackage{booktabs} 
\usepackage{array} 
\usepackage{paralist} 
\usepackage{verbatim} 
\usepackage{mathrsfs}
\usepackage{amssymb}
\usepackage{amsthm}
\usepackage{amsmath,amsfonts,amssymb}
\usepackage{esint}
\usepackage{graphics}
\usepackage{enumerate}
\usepackage{mathtools}
\usepackage{xfrac}
\usepackage{subcaption}
\usepackage{stmaryrd}
 \usepackage{mathabx}

\usepackage[dvipsnames]{xcolor}
\usepackage[colorlinks=true, pdfstartview=FitV, linkcolor=blue, citecolor=blue, urlcolor=blue]{hyperref}
\usepackage[normalem]{ulem}

\usepackage{tikz}
\usetikzlibrary{calc}
\usepackage{pgf}
\usetikzlibrary{external}
\numberwithin{equation}{section}
\numberwithin{figure}{section}

\newtheorem{theorem}{Theorem}[section]

\newtheorem{assumption}[theorem]{Assumption}
\newtheorem{corollary}[theorem]{Corollary}
\newtheorem{proposition}[theorem]{Proposition}
\newtheorem{lemma}[theorem]{Lemma}
\theoremstyle{definition}
\newtheorem{definition}[theorem]{Definition}

\newcommand*{\supp}{\ensuremath{\mathrm{supp\,}}}

\newcommand*{\N}{\ensuremath{\mathbb{N}}}

\newcommand*{\Z}{\ensuremath{\mathbb{Z}}}

\newcommand*{\R}{\ensuremath{\mathbb{R}}}

\newcommand{\eps}{\varepsilon}

\renewcommand*{\tilde}{\widetilde}

\newcommand{\ep}{\eps}

\DeclareSymbolFont{boldoperators}{OT1}{cmr}{bx}{n}
\SetSymbolFont{boldoperators}{bold}{OT1}{cmr}{bx}{n}
\usepackage{accents}
\newcommand\thickbar[1]{\accentset{\rule{.45em}{.6pt}}{#1}}
\renewcommand{\bar}{\thickbar}

\newcommand{\T}{\mathbb{T}}

\def\XXint#1#2#3{{\setbox0=\hbox{$#1{#2#3}{\int}$}
\vcenter{\hbox{$#2#3$}}\kern-.5\wd0}}

\let\originalleft\left
\let\originalright\right
\renewcommand{\left}{\mathopen{}\mathclose\bgroup\originalleft}
\renewcommand{\right}{\aftergroup\egroup\originalright}

\newcommand{\indc}{\mathds{1}}

\newcommand{\C}{\mathbb{C}}

\renewcommand{\hat}{\widehat}

\makeatletter
\pgfmathdeclarefunction{erf}{1}{%
  \begingroup
    \pgfmathparse{#1 > 0 ? 1 : -1}%
    \edef\sign{\pgfmathresult}%
    \pgfmathparse{abs(#1)}%
    \edef\x{\pgfmathresult}%
    \pgfmathparse{1/(1+0.3275911*\x)}%
    \edef\t{\pgfmathresult}%
    \pgfmathparse{%
      1 - (((((1.061405429*\t -1.453152027)*\t) + 1.421413741)*\t 
      -0.284496736)*\t + 0.254829592)*\t*exp(-(\x*\x))}%
    \edef\y{\pgfmathresult}%
    \pgfmathparse{(\sign)*\y}%
    \pgfmath@smuggleone\pgfmathresult%
  \endgroup
}
\makeatother

\usepackage{titlesec}

\newcommand{\addperiod}[1]{#1.}
\titleformat{\section}
   {\centering\normalfont\Large}{\thesection.}{0.5em}{}
\titleformat*{\subsection}{\bfseries}
\titleformat{\subsubsection}[runin]
  {\normalfont\bfseries}
  {\thesubsubsection.}
  {0.5em}
  {\addperiod}
\titleformat*{\subsubsection}{\normalfont\itshape}
\titleformat*{\paragraph}{\bfseries}
\titleformat*{\subparagraph}{\large\bfseries}

\title{Superexponential dissipation enhancement on $\T^d$}

\author{
Keefer Rowan\thanks{Courant Institute of Mathematical Sciences,  New York University.
{\footnotesize \href{mailto:keefer.rowan@cims.nyu.edu}{keefer.rowan@cims.nyu.edu}.}
}
}
\date{\today}

\usepackage[nottoc,notlot,notlof]{tocbibind}

\begin{document}

\maketitle

\begin{abstract}
    We construct incompressible velocity fields that exhibit faster than exponential dissipation for particular solutions to the advection-diffusion equation on $\T^d$. In 2D, we construct a velocity field in $L^\infty_{t,x}$ and exhibit a solution that decays with double exponential rate $e^{-C^{-1} e^{C^{-1}t}}$. In 3D, we construct a velocity field in $L^\infty_t W^{1,\infty}_x$ and exhibit a solution that decays with rate $e^{-C^{-1} t^2}$. In 4D, we construct a velocity field in $L^\infty_t C^\infty_x$ and exhibit a solution that decays with \textit{some} superexponential rate.
\end{abstract}


\section{Introduction}

We consider the decay of solutions to the advection-diffusion equation on $\T^d.$ That is functions $\theta : [0,\infty) \times \T^d \to \R$ solving
\begin{equation}
    \label{eq:advection-diffusion}
    \dot \theta_t = \Delta \theta_t + u_t \cdot \nabla \theta_t,
\end{equation}
for an advecting flow $u : [0,\infty) \times \T^d \to \R^d$ that we always take to be incompressible: $\nabla \cdot u=0.$

We are interested in how fast solutions to~\eqref{eq:advection-diffusion} can decay. The norm we use to measure the size of $\theta$ will always be $L^2(\T^d)$, though we note that by parabolic smoothing, rates in other norms will be comparable (see also~\cite{blumenthal_norm_2023}). Exponential decay for zero-mean solutions is direct from the energy identity~\eqref{eq:energy-identity} and the Poincar\'e inequality; we wish to exhibit solutions that decay \textit{faster than exponentially}. Faster than exponential decay is known for the advection-diffusion equation on non-compact domains and using unbounded velocity fields---e.g.\ by Couette flow on $\R \times \T$ or $\R^2$---as well for the time-discretized version of the advection-diffusion equation in which one alternates precomposition by a diffeomorphism $\T^d \to \T^d$ with pure diffusion---e.g.\ by Arnold's cat map. These cases are discussed further in Section~\ref{ss:background}. 

Our results are the first to address the case of advection-diffusion on $\T^d$. We give a sequence of results, each one giving more regularity on the advecting flow. We note however that as the regularity is increased, the rate of decay---while continuing to be superexponential---gets worse, and we are required to work in higher dimension.

Our first result is for velocity fields uniformly bounded in space and time on $\T^2.$

\begin{theorem}
\label{thm:2d}
    For $d=2$, there exists $C>0$, a velocity field $u \in L^\infty([0,\infty) \times \T^2)$ with $\nabla \cdot u =0$, and a choice of initial data $\theta_0 \ne 0, \theta_0 \in C^\infty(\T^2)$, so that if $\theta_t$ is the solution to~\eqref{eq:advection-diffusion}, then
    \begin{equation}
    \label{eq:2d bound}
    \|\theta_t\|_{L^2(\T^2)} \leq C e^{- C^{-1} e^{C^{-1}t}} \|\theta_0\|_{L^2(\T^2)}.
    \end{equation}
\end{theorem}

The double exponential rate of~\eqref{eq:2d bound} is essentially optimal: as discussed in Section~\ref{sss:lower bounds and decay} there exists a double exponential lower bound on $\|\theta_t\|_{L^2}$. Our next result addresses the case of a Lipschitz velocity field, which is often the relevant regularity for enhanced dissipation phenomena. We note however that our rate is no longer double exponential and is likely suboptimal.

\begin{theorem}
\label{thm:3d}
    For $d = 3$, there exists $C>0$, a velocity field $u \in L^\infty([0,\infty), W^{1,\infty}(\T^3))$ with $\nabla \cdot u =0$, and a choice of initial data $\theta_0 \ne 0, \theta_0 \in C^\infty(\T^3)$, so that if $\theta_t$ is the solution to~\eqref{eq:advection-diffusion}, then
    \[\|\theta_t\|_{L^2(\T^3)} \leq C e^{- C^{-1} t^2} \|\theta_0\|_{L^2(\T^3)}.\]
\end{theorem}

Our final result is for a smooth velocity field, uniformly bounded in $C^n(\T^d)$ for all $n$. The rate $f(t)$ given below is non-explicit. Some additional work should be able to yield an explicit subalgebraic rate, but since we believe that this is likely far from optimal in any case, we give a qualitative statement in order to simplify the proof.

\begin{theorem}
\label{thm:4d}
    For $d =4$, there exists a velocity field $u : [0,\infty) \times \T^4 \to \R^d$ with $\nabla \cdot u =0$ and a choice of initial data $\theta_0 \ne 0, \theta_0 \in C^\infty(\T^4)$, so that for any $n \geq 0$, $u \in L^\infty([0,\infty), W^{n,\infty}(\T^d))$ and if $\theta_t$ is the solution to~\eqref{eq:advection-diffusion}, then
    \[\|\theta_t\|_{L^2(\T^4)} \leq e^{-  f(t)t}\|\theta_0\|_{L^2(\T^4)},\]
    for some increasing function $f : [0,\infty) \to [0,\infty)$ such that $\lim_{t \to \infty} f(t) = \infty$.
\end{theorem}

\subsection{Background and discussion}
\label{ss:background}

The phenomenology of solutions to advection-diffusion equations---also know as passive scalars---is both of independent interest and a stepping stone to understanding the decidedly more difficult phenomenology of fluids. While passive tracers have long been of interest in the physics community (see~\cite{falkovichParticlesFieldsFluid2001} for a thorough overview), it is only more recently that they have seen extensive mathematical treatment. In the rough setting, the intriguing phenomenon of anomalous dissipation has been thoroughly investigated, see~\cite{drivas_lagrangian_2017,colombo_anomalous_2023,armstrong_anomalous_2025,burczak_anomalous_2023,elgindi_norm_2024,rowan_anomalous_2024,johansson_anomalous_2024,hess-childs_turbulent_2025}, among others. However, for our purposes, the more relevant literature is on \textit{mixing}, \textit{enhanced dissipation,} and their interrelations. 
\subsubsection{Mixing and enhanced dissipation}

Mixing is (most naturally) a phenomenon of the pure transport equation: 
\[\dot \theta_t = u_t \cdot \nabla \theta_t,\]
where we still consider $u_t$ to be incompressible $\nabla \cdot u_t =0.$ The advective term $u_t \cdot \nabla \theta_t$ tends to generate small scales in the solution $\theta_t$, moving mass from low wavenumber Fourier modes to high wavenumber Fourier modes. Mixing happens when \textit{all of the Fourier mass} is moved to high wavenumbers, which can be easily quantified using negative regularity Sobolev norms. That is (loosely) $u_t$ exhibits mixing when $\|\theta_t\|_{H^{-1}} \to 0$ as $ t \to \infty.$ Note that as $\nabla \cdot u_t=0,$ $\|\theta_t\|_{L^2} = \|\theta_0\|_{L^2}$, and as such the phenomenon of mixing necessarily only happens if the Fourier mass moves from low modes to high modes. A great number of flows $u_t$ have been proven to exhibit mixing:~\cite{yao_mixing_2017,alberti_exponential_2019,elgindi_universal_2019,bedrossian_almost-sure_2022,myers_hill_exponential_2022,blumenthal_exponential_2023,cooperman_exponential_2024,elgindi_optimal_2025,navarro-fernandez_exponential_2025} among others.

Enhanced dissipation is a phenomenon of the advection-diffusion equation, most easily understood through the introduction of an explicit diffusion coefficient:
\[\dot \theta_t = \kappa \Delta \theta_t + u_t \cdot \nabla \theta_t.\]
When $\kappa =0$, since $\nabla \cdot u_t =0$, we have that $\|\theta_t\|_{L^2} = \|\theta_0\|_{L^2}$. For $\kappa>0$, the Poincar\'e inequality ensures that any zero-mean data decays with an exponential rate
\[\|\theta_t\|_{L^2} \leq e^{-C^{-1} \kappa t}\|\theta_0\|_{L^2}.\]
When $u_t =0$, this rate is saturated e.g.\ by taking $\theta_0(x) = \sin(x_1)$. Enhanced dissipation happens when the presence of the advecting flow $u_t$ causes dissipation with a faster rate, asymptotically as $\kappa \to 0$, than is present for the pure diffusive equation. For example, instead of the exponential rate going to $0$ as $O(\kappa)$, it could rather be $O(\kappa^{\alpha})$ for some $\alpha \in(0,1)$ or even $O((\log \kappa^{-1})^{-1})$. Dissipation enhancement has also been demonstrated in myriad settings:~\cite{constantin_diffusion_2008,bedrossian_enhanced_2017,zelati_stochastic_2021,bedrossian_almost-sure_2021,albritton_enhanced_2022,coti_zelati_enhanced_2023,villringer_enhanced_2024,elgindi_optimal_2025,zelati_stochastic_2025} among others. Mixing and enhanced dissipation are in fact intimately connected. In particular, mixing implies dissipation enhancement---a relationship that can even be made quantitative~\cite{constantin_diffusion_2008,feng_dissipation_2019,zelati_relation_2020,cooperman_exponentially_2025}.

\subsubsection{Lower bounds and superexponential decay}
\label{sss:lower bounds and decay}
The above results on enhanced dissipation give upper bounds on the $L^2$ norm that decay exponentially fast. We know that decay with an exponential rate is optimal for the purely diffusive case $u_t=0$, as is made clear by the explicit form of solutions to the heat equation. However, once we add in the advection, it is less clear what the optimal rate of decay is. When the advective term is time-independent: $u_t =u_0$, spectral theory considerations ensure we continue to only have exponential decay of solutions and no faster; similarly in the time periodic case $u_{t+1} = u_t$. However in the truly time-inhomogeneous case, the fastest possible rate of decay is less clear.

There are however some lower bounds available even in the time-inhomogeneous setting.~\cite{miles_diffusion-limited_2018}, adapting arguments of~\cite{poon_unique_1996}, give double exponential \textit{lower bounds} for $\|\theta_t\|_{L^2}$ in case that $u \in L^\infty_t W^{1,\infty}_x$ and even in the case that $u \in L^\infty_{t,x}$, though in the latter setting there are constants that diverge as $\kappa \to 0.$ Thus for the advection-diffusion equation on the torus, there was a large gap between the fast known decay rate (exponential) and the best known lower bound (double exponential). It is the goal of this work to (partially) bridge this gap by exhibiting faster than exponentially decaying solutions.

In some settings, faster than exponential decay is already known. If one considers the advection-diffusion equation on $\R^d$ and with a linear (hence a necessarily unbounded) advecting flow, $u(x) = Ax$, then taking a Fourier transform of the advection-diffusion equation gives a particular simple first order PDE that admits an explicit solution through characteristics. We note that we must have $\mathrm{Tr}\, A =0$ to enforce $\nabla \cdot u =0$. Taking
\[A = \begin{pmatrix} 1 & 0 \\0 &-1\end{pmatrix},\]
one can readily verify that for well chosen data we get a double exponential rate of decay 
\[\|\theta_t\|_{L^2} \leq Ce^{-C^{-1} e^{C^{-1} t}} \|\theta_0\|_{L^2}.\]
We can also consider the Couette flow (for which the analysis could also be performed on $\R \times \T$ in place of $\R^2$), given by
\[A = \begin{pmatrix} 0 & 0 \\ 1 & 0 \end{pmatrix}.\]
One can then compute (again for well chosen data) that we get the super exponential rate of decay
\[\|\theta_t\|_{L^2} \leq C e^{- C^{-1}t^3}\|\theta_0\|_{L^2}.\]

The case of linear advecting flows is very special for a variety of reasons, but the most important for our considerations is that the advection-diffusion equation with a linear advecting flow takes pure Fourier modes to pure Fourier modes. This makes the interplay between the advection and the diffusion much simpler and allows for explicit computation of decay rates.

Working on a  compact domain such as $\T^d$, we no longer have access to linear advecting flows. However, we can consider the following time-discretized variation of the advection-diffusion equation instead. Take some volume-preserving diffeomorphisms $\Phi_j: \T^d \to \T^d$ and inductively define $\theta_n$ as
\[\theta_{n+1} := e^{\Delta} \theta_n \circ \Phi_{n+1}.\]
We note that this variation of the advection-diffusion generalizes the ``pulsed diffusion'' problem, in which one alternates flowing by the pure transport equation and by the pure diffusion. It is however a strict generalization as we allow maps $\Phi_j$ that are not generated as the unit-time map of a flow of diffeomorphism. In particular, we can take $\Phi_j(x) = Ax$ for some $A \in SL(\Z^d)$ to be a linear toral automorphism, such as Arnold's cat map. In this case, we see the map
\[\theta \mapsto e^{\Delta} \theta \circ \Phi \]
also takes Fourier modes to Fourier modes, once again allowing for explicit computation of decay rates in $L^2$. For $A$ hyperbolic---e.g.\ Arnold's cat map---we again get a double exponential rate of decay:
\[\|\theta_n\|_{L^2} \leq C e^{- C^{-1} e^{C^{-1} n}}\|\theta_0\|_{L^2}.\]
See~\cite{feng_dissipation_2019} for further discussion of superexponential dissipation in this discrete-time setting.

\subsubsection{Batchelor scale formation}

The above results may give the impression that faster than exponential decay for solutions to the advection-diffusion equation on $\T^d$ is largely expected. However, there are good reasons to think such fast decay is strongly disfavored. We recall the standard energy identity for the advection-diffusion equation, keeping the explicit diffusion coefficient $\kappa$:
\[\frac{d}{dt} \|\theta_t\|_{L^2} = -\kappa \frac{\|\nabla \theta_t\|_{L^2}^2}{\|\theta_t\|_{L^2}^2}\|\theta_t\|_{L^2}.\]
Thus the (instantaneous) rate of decay in $L^2$ is given by the quantity
\[-\kappa \frac{\|\nabla \theta_t\|_{L^2}^2}{\|\theta_t\|_{L^2}^2}.\]
We note that $\frac{\|\nabla \theta_t\|_{L^2}}{\|\theta_t\|_{L^2}}$
can be thought of as a natural inverse length scale for the passive scalar $\theta_t$, giving the typical wavenumber at which $\theta_t$ has Fourier mass. 

As is seen in the numerical simulations of the pioneering study~\cite{miles_diffusion-limited_2018}, fluctuations of the passive scalar beyond the ``diffusive scale'' $\kappa^{1/2}$ tend to be strongly suppressed by the diffusion, causing a tendency of $\frac{\|\nabla \theta_t\|_{L^2}}{\|\theta_t\|_{L^2}}$ to saturate at $O(\kappa^{-1/2}),$ with $\frac{\|\nabla \theta_t\|_{L^2}}{\|\theta_t\|_{L^2}} \gg \kappa^{-1/2}$ being highly unstable. The length scale $\kappa^{1/2}$ is sometimes called the Batchelor scale after~\cite{batchelor_small-scale_1959}. As a consequence to the saturation at the Batchelor scale, we have that the exponential rate of decay, $\kappa \frac{\|\nabla \theta_t\|_{L^2}^2}{\|\theta_t\|_{L^2}^2}$, saturates at $O(1).$ In fact~\cite{miles_diffusion-limited_2018} conjectures that there is a universal upper bound on $\kappa \frac{\|\nabla \theta_t\|_{L^2}^2}{\|\theta_t\|_{L^2}^2}$ and so that faster than exponential decay is impossible. Our results show this conjecture to be false.

The above discussion suggests that we should not at least \textit{typically} expect faster than exponential decay, which would require $\frac{\|\nabla \theta_t\|_{L^2}}{\|\theta_t\|_{L^2}}  \to \infty$. This intuition is confirmed by~\cite{hairer_lower_2024}, which takes the advecting flow $u_t$ to be the solution to the stochastically forced Navier--Stokes equations on $\T^2$ and demonstrates that typical data decays with at most an exponential rate.

This discussion appears to be in conflict with the above examples which demonstrate superexponential dissipation: if $\frac{\|\nabla \theta_t\|_{L^2}}{\|\theta_t\|_{L^2}} \gg \kappa^{-1/2}$ is so unstable, how do those examples exhibit $\frac{\|\nabla \theta_t\|_{L^2}}{\|\theta_t\|_{L^2}} \to \infty$? The answer is in the special fact that those constructions rely on taking pure Fourier modes to pure Fourier modes. The  diffusion enforces $\frac{\|\nabla \theta_t\|_{L^2}}{\|\theta_t\|_{L^2}} \approx \kappa^{-1/2}$ by suppressing small scales relative to large scales, causing the amount of mass on very small scales to decrease so that the majority of mass lives on wavenumbers $\lesssim \kappa^{-1/2}.$ By taking Fourier modes to Fourier modes, the above examples keep Fourier mass very concentrated. As such, there is no mass on large scales that the small scales can be suppressed relative to, allowing for ever smaller scales to form. Moving mass from pure Fourier modes to pure Fourier modes is also the central trick of our construction; it is however much less clear how to do this in the setting of the advection-diffusion equation on $\T^d$ which doesn't allow for the linear flows or linear maps considered above.

\subsubsection{Other lower bounds on the advection-diffusion equation}

We note that beyond the double exponential lower bounds given by~\cite{miles_diffusion-limited_2018}, there are some additional lower bounds available for the advection-diffusion equation.~\cite{seis_bounds_2023} gives precise and sharp bounds on the exponential rate of enhanced dissipation and the size of the multiplicative prefactor. These bounds are saturated at some finite time (e.g.\ $t \approx \log \kappa^{-1}$) while we are interested in the very long time behavior $t \to \infty.$ \cite{nobili_lower_2022} gives algebraic-in-time lower bounds on solutions to the advection-diffusion equation on $\R^d$ given space and time decay conditions on the advecting flow.

\subsection{Open problems}

The most direct next question stemming from this work is whether one can construct a ``universal'' superexponentially dissipating flow---that is a flow on $\T^d$ that causes superexponential dissipation for all mean-zero data. Such a flow would be a counterexample to a variety of conjectures one could make relating to the formation of a Batchelor scale, that is $\frac{\|\nabla \theta_t\|_{L^2}}{\|\theta_t\|_{L^2}}$ saturating at $\kappa^{-1/2}.$

One can also try to improve on the rates attained in the higher regularity settings of Theorem~\ref{thm:3d} and Theorem~\ref{thm:4d}, as well as the dimension requirements. For example, the following problem remains open: does there exist a uniformly smooth advecting flow on $\T^2$ that causes double exponential decay in the advection-diffusion equation?

A line of inquiry of greater physical interest is understanding Batchelor scale formation in generic flows, such as the sufficiently non-degenerate random advecting flows considered in~\cite{bedrossian_almost-sure_2022, blumenthal_exponential_2023}.~\cite{hairer_lower_2024} is a meaningful step in this direction, however while their results rule out superexponential dissipation, they do not get the expected quantitative behavior of $\frac{\|\nabla \theta_t\|_{L^2}}{\|\theta_t\|_{L^2}}$ as one sends $\kappa \to 0$, getting a slightly too large upper bound of $\frac{\|\nabla \theta_t\|_{L^2}}{\|\theta_t\|_{L^2}} \lesssim \kappa^{-2}$, allowing for potentially faster decay rates for smaller $\kappa.$ Additionally,~\cite{hairer_lower_2024} relies on a highly non-degenerate advecting flow with stochastic fluctuations present on every Fourier mode. It would be very interesting to see lower bounds on rates of decay for typical data that 1) have the correct asymptotics as $\kappa \to 0$ and 2) hold for less noisy flows, such as the Pierrehumbert mixing example~\cite{pierrehumbert_tracer_1994}.

\section{Notation, overview of the argument, and a reduction of the problem}

Before going further, let us specify our conventions. We identify $\T^d := [0,2\pi]/\sim$. For $k \in \Z^d$, we denote the Fourier mode
    \[f_k(x) := e^{i k \cdot x}.\]
For $\theta : \T^d \to \C$, we denote $\hat \theta : \Z^d \to \C$ so that
    \[\theta = \sum_k \hat \theta(k) f_k \quad \text{and} \quad \hat \theta(k) = \frac{1}{(2\pi)^d} \int e^{-i k\cdot x} \theta(x)\,dx. \]

\subsection{Overview of the argument}

\label{ss:overview}
We first note that it suffices to prove the result for complex-valued initial data to give the result for real-valued data. Since the equation is reality-preserving, if some complex solution decays superexponentially, the real (and imaginary) parts of the solution are also superexponentially decaying solutions. It will be useful to take complex solutions so we can work with single Fourier modes. 

The main idea of the argument is construct advecting flows $u_t$ that take some initial data $\theta_0$ supported on a single Fourier mode and, flowing under the advection-diffusion equation~\eqref{eq:advection-diffusion}, take $\theta_0$ to $\theta_T$, where $T$ is some finite time and $\theta_T$ is supported on a single Fourier mode of larger wavenumber. If we can accomplish this for a sufficiently rich collection of initial and final Fourier mode wavenumbers, we can iterate the procedure in order to construct a flow $u_t$ that send $\theta_t$ to infinity in Fourier space, hence causing $\frac{\|\nabla \theta_t\|_{L^2}}{\|\theta_t\|_{L^2}} \to \infty$ and so causing faster than exponential decay.

The rate at which we are able to transfer the solution $\theta_t$ from one Fourier mode to another, as well as the separation of the wavenumber magnitudes of the Fourier modes, determines the rate of growth of $\frac{\|\nabla \theta_t\|_{L^2}}{\|\theta_t\|_{L^2}},$ and hence the superexponential rate of decay. We discuss below why the different rates appear in the different settings of Theorem~\ref{thm:2d}, Theorem~\ref{thm:3d}, and Theorem~\ref{thm:4d}.

The main technical issue then is building the velocity field that transfers the solution from one pure Fourier mode to another. We first give a complete overview of how we construct such a flow for Theorem~\ref{thm:2d} and then note what new ideas are needed for Theorem~\ref{thm:3d} and Theorem~\ref{thm:4d}.

\subsubsection{Reduction to a 1D problem}

The first, most basic idea is that we can get a dimensional reduction of the problem (in any dimension) to an essentially 1D problem, which will be studied in Fourier space and as such will be an infinite family of ODEs indexed by $k \in \Z$. This dimensional reduction is achieved as follows. Suppose we are starting on the Fourier mode $f_a$ and want to move all the mass onto the mode $f_{a+b}$ with $a,b \in \Z^d -\{0\}$. We can then consider a flow $u_t(x) = v_t(b \cdot x)$ for some $v : [0,\infty) \times \T \to \R^d$. If we take a flow of this form and have initial data $\theta_0 = f_a$, then the solution $\theta_t$ will be supported in Fourier space exclusively on $a + kb, k \in \Z$. It only will be through flows of this form that we transport our Fourier mass from one mode to another, and as such what we need is a good understanding of the infinite ODE system for the coefficients for $f_{a+kb}$ that this induces.

Proposition~\ref{prop:ode-to-pde} codifies the translation of the problem into studying an infinite ODE system. It gives that to build a flow $u_t$ that moves Fourier mass from $a$ to $a+b$ for $a,b \in \Z^d - \{0\}$ it suffices build $v^j : [0,\infty) \to \C$ for $j \in \Z$ such that if we solve
\[\begin{cases}
    \dot z^k_t =  - d_k z^k_t  + i\sum_j v^j_t  z^{k-j}_t,\\
    z^k_0 = \delta_{k,0},
\end{cases}
\]
then there is some time $T>0$ such that $z^k_T = \beta \delta_{k,1}$ for some $\beta \in \C$, and where\footnote{You may note that Proposition~\ref{prop:ode-to-pde} there are additional parameters $A,L$. These are to clean up the technical presentation but are inessential, so we ignore them for now.}
\[d_k :=|a+kb|^2.\]
In order to enforce that $u_t$ is real-valued, we need that $v^{-j}_t = \overline{v^j_t},$ as such we are really only free to choose $v^j_t$ for $j \in \N.$

\subsubsection{What modes we can hope to move mass between}

We now ask for what values $a,b \in \Z^d - \{0\}$ we can reasonably hope to be able to move the mass from $a \mapsto a+b$. The simplest guess would be that we can take $a = (r,0)$ and $b = (1,0)$ so we are moving mass from $(r,0) \mapsto (r+1,0)$. The divergence-free constraint of $u_t$ however makes this impossible (at least to do in a single step). When converting the advection-diffusion equation into the ODE system for $z^k_t$, we pick up a prefactor on the $v^j_t$ that is $0$ if $b$ is parallel to $a$ and $1$ if $b$ is perpendicular to $a$. As such, we want to pick $b$ to be close to perpendicular to $a$, making the above choice of $a=(r,0), b=(1,0)$ a bad one.

The next guess would be to take $a =(r,0)$ and $b = (1,1)$ so we move mass $(r,0) \mapsto (r+1,1)$. Then taking $b= (1,-1)$, we could hope to move mass to $(r+2,0)$ and then iterate the process. This however is also a poor choice. The reason is the diffusion coefficients $d_k$ that appear in the ODE system. Let $a = (r,0), b=(1,1)$, suppose $r$ is large, and consider $k \approx -r/2.$ In that case the diffusion coefficient $d_k \approx r^2/2 \ll d_0,d_1 \approx r^2.$ As such, once a little bit of mass moves onto the modes $z^k_t$ for $k \approx -r/2$, the diffusion will damp the $z^0_t, z^1_t$ modes much more strongly than the $z^k_t$ modes, and so the majority of the mass will quickly accumulate on the $z^k_t$ modes. In short order, all but a vanishing percentage of the mass of $z_t$ will be on the $z^k_t$ modes for $k \approx -r/2.$ This goes exactly contrary to our goals: we have moved the majority of mass from higher Fourier modes to lower Fourier modes.

The take away is the following. If we want to choose $a,b \in \Z^d - \{0\}$ so that we can move mass from $a \mapsto a+b$, then we need that 1) $b$ is far from parallel to $a$ and 2) that for $d_k = |a+kb|^2$ we have that $d_k \geq d_0,d_1$ for $k \ne 0,1$. In fact, it will be advantageous to strengthen the second condition to $d_k \gg d_0,d_1$ for $k \ne 0,1$. This will then imply that the diffusion strongly suppresses all modes except $z^0_t,z^1_t$, allowing us to treat all other terms as perturbations.

With the above discussion in mind, in the setting of Theorem~\ref{thm:2d} a good choice is to take $a = (r,0)$ to $a+b = (0,r+1)$ for $r \in \N$ sufficiently large. By iterating the map $(r,0) \mapsto (0,r+1)$, appropriately swapping $x$ and $y$ on every step, we see that this allows us to send the Fourier mass to $\infty$. We see also that $b$ is essentially at $45^\circ$ angle to $a$ so far from parallel. Additionally, we compute $d_0 =r_2, d_1 = r^2 + 2r + 1$, and for $k \ne 0,1, d_k \geq 4r^2 \gg d_0,d_1$.

\subsubsection{\texorpdfstring{First step: moving the mass off of $z^0_t$}{First step: moving the mass off of z\^0\_t}}

The first step of the argument is to choose $v^j_t$ so that at some time $T>0$, we have that $z^1_T \ne  0$ and $z^0_T =0,$ that is we want to move all of the mass off of $k=0.$ We note that this will leave mass scattered on the other modes---$z^k_T$ for $k \ne 0,1$---as well; dealing with this scattering of mass will be the goal of the next step. Moving the mass off of $z^0_t$ requires some bookkeeping but is not too difficult; it is however in this step where the regularity of the velocity field $u_t$ is determined. This is because, ignoring $k \ne 0,1$ entirely for now, we want to move the mass from $z^0_t$ to $z^1_t$. This is accomplished through just using $v^1_t$ (and also $v^{-1}_t$, since $v^{-1}_t = \overline{v^1_t}$). Looking at the ODE system, we see the rate at which we can move mass off of $z^0_t$ is determined by the amount of mass on $z^1$ as well as the magnitude of $v^1_t$. Just looking at the simplified system of $z^0_t,z^1_t$, one can verify that there is a minimum magnitude of $v^1_t$---depending on $d_1 - d_0$---needed to be able to move all of the mass off of $z^0_t$: if $v^1_t$ is too small, some mass will always remain trapped on $z^0_t.$ Computing this minimum magnitude in the 2D setting, one sees that to move from the Fourier mode $f_{(r,0)}$ to $f_{(0,r+1)}$ we require a velocity field $u_t$ such that $|\hat u_t(-r,r+1)| = O(1).$ This is why we only get $L^\infty$ regularity in the 2D setting: we require ever larger Fourier modes of $u_t$ to be $O(1)$, something that is only uniformly bounded in a zero regularity norm such as $L^\infty.$

In fact, doing the computation in greater generality, we can see that---assuming $b$ is far from parallel to $a$---in order to move the mass from $a$ to $a+b$, we need to build $u_t$ so that $|\hat u(b)| \approx \frac{D}{|a|}$ where $D := |a+b|^2 - |a|^2$. We note that in the setting of $a = (r,0)$ and $b = (-r,r+1)$ that we are considering presently, $\frac{D}{|a|} \approx 1.$ One can also see that the time required to move the mass is $\approx \frac{1}{|\hat u(b)| |a|}$.

\subsubsection{Second step: cleaning up the error}

Once we have moved all of the mass off of $z^0_t$, the next step is to turn off the advection---setting $v^j=0$---and let the diffusion alone act. The diffusion will preserve that $z^0_t=0$ and since $d_k \gg d_1$ for $k \ne 0,1$, waiting long enough will cause almost all of the mass of $z^k_t$ to be in the $k=1$ coordinate, that is 
\[\frac{\sum_{k \ne 1} |z^k_t|^2}{\sum_k |z^k_t|^2} \ll 1.\]
We then want to choose $v^j_t$ that takes the error---the mass not on $z^1_t$---and sends it to $0$. This has the form of a perturbative control problem, as the error we are seeking to remove is perturbatively small. This error is sent to $0$ with what is essentially a Newton's method, unfolding dynamically in time. That is, we expand the problem perturbatively and seek to remove the leading order error (of order $\ep$), requiring us to solve a linear control problem. Canceling the error to leading order then leaves behind an error of lower order (order $\ep^2$). Repeating this process on a sequence of dyadic time intervals leads to convergence of the error to $0$ in a finite (unit) time interval, provided the initial error was small enough. This then completes the task, at the final time $z^k_t = \beta \delta_{k,1}.$

Thinking of the problem in real space (as opposed to Fourier space), it is clear we must somehow be using the diffusion in order to remove the error in this step. That is because, without the diffusion, there are geometric obstructions to building a flow that takes us from \textit{close to all of the mass} being on $z^1_t$ to \textit{all of the mass} being on $z^1_t$, since flow by a divergence-free velocity field preserves the measure of the super-level sets of the solution. Thus we must be using a somewhat delicate interplay between the advection and the diffusion in order to remove the final error.

This interplay reveals itself in the linear control problem that we must solve for the Newton's method. The linear control problem almost entirely decouples mode-by-mode; that is we can choose $v^k_t$ in order to remove the error on the mode $z^{k+1}_t$, without affecting to leading order the errors on any other modes. However, because $v^{-k}_t = \overline{v^k_t},$ we end up needing to solve the problems for $z^{1-k}_t$ and $z^{1+k}_t$ simultaneously. This in principle should not be a problem, since we can choose from the infinite dimensional space of time dependent paths $v^k : [0,1] \to \C$. However in the absence of diffusion, the time dependency trivializes: the only thing that enters into the final computation is $\int_0^1 v^k_s\,ds$, a single complex number. We would then have a space of controls $\C^1$, but we would need to fix errors $(z^{1-k},z^{1+k}) \in \C^2,$ which is generically impossible. This is where the diffusion comes into play. The diffusion---in particular the difference of the diffusion coefficients $d_{1+k}$ and $d_{1-k}$---induce a nontrivial time dependency into the effect of the control $v^k_t$. This then means, provided that $d_{1+k} \ne d_{1-k}$, we truly do have an infinite dimensional space of controls in order to solve for an error in $\C^2$. Thus the presence of the diffusion allows for the solvability of the linear control problem.

In terms of rates and regularity, this perturbative control step is much less important. That is, it can always (both in the setting of Theorem~\ref{thm:2d} as well as the settings of Theorem~\ref{thm:3d} and Theorem~\ref{thm:4d}) be done more quickly and at higher regularity than the first step of moving the mass off of $z^0_t.$ This is essentially due to the fact that the diffusion acts very strongly to rapidly smooth out the problem.

\subsubsection{Rate of decay for Theorem~\ref{thm:2d}}

As noted above, the time to move the mass $(r,0) \mapsto (0,r+1)$ is dominated by the first step, which takes time $\approx \frac{1}{|u(b)| |a|} \approx \frac{1}{r}$. This then gives that we can moves the Fourier mass moves to infinite wavenumber with an exponential rate. We thus have that the exponential rate of decay grows exponentially, giving the double exponential decay of Theorem~\ref{thm:2d}.

\subsubsection{Modifications for Theorem~\ref{thm:3d}}

Let us now discuss what we need to change for Theorem~\ref{thm:3d}. Our goal is to improve the regularity of the advecting flow $u_t$ while still moving the Fourier mass mode to mode out to $\infty.$ As discussed above, the step which limits the regularity was the first step of the argument, where all of the mass is moved off of the smallest wavenumber Fourier mode. There are two options for improving the regularity of the advecting flow: make $b$ smaller or make $|\hat u(b)|$ smaller. In order to have the right structure on the diffusion coefficients $d_k$---$d_k \gg d_0,d_1$ for $k \ne 0,1$---it is helpful to take $|b| \approx |a|$, essentially so that $a+b$ can be at a meaningfully different angle than the vector $a$. Thus we rather focus on making $|\hat u(b)|$ smaller. It was noted above that we need $|\hat u(b)| \approx \frac{D}{|a|}$ where $D = |a+b|^2 - |a|^2.$ Clearly, $D \geq 1$, since it is a difference of integers (which we want to be positive as we want to move mass to larger wavenumbers). Thus the smallest we can hope to make $|\hat u(b)| \approx \frac{1}{|a|}$. Given that we also want to take $|b| \approx |a|$, this is exactly consistent with the Lipschitz regularity of Theorem~\ref{thm:3d}. 

In order to have $|\hat u(b)| \approx \frac{1}{|a|}$, we need $D \approx 1$: that is we need $a, b$ such that $|a+b|^2 - |a|^2 \approx 1.$ What we need to then understand is the gaps between square magnitudes on the lattice $\Z^d$, that is the gaps in the set $S_d:= \{|k|^2 : k \in \Z^d\} \subseteq \N$. In 2D, there are arbitrarily large gaps: for all $R>0$ there exists $n \in S_2$ such that $(n, n+R) \cap S_2 = \emptyset$. This is no longer the case in 3D; Legendre's three square theorem in particular implies that every number that is $1 \pmod 4$ is the sum of three squares. As such, we have that the gaps in $S_3$ are at most size $4.$

It is for this reason we are forced to work in 3D for Theorem~\ref{thm:3d}. We also need some elementary geometry to ensure that we can both choose $b$ so that $|a+b|^2 - |a|^2 \approx 1$ while still maintaining that $b$ isn't close to parallel with $a$ and $d_k \gg d_0,d_1$ for $k \ne 0,1.$ However, once this is concluded, the argument sketch above for Theorem~\ref{thm:2d} applies and allows us to move mass across the Fourier modes $a \mapsto a+b$ with a uniformly Lipschitz velocity field and in time $\frac{1}{|u(b)||a|} \approx 1.$ Thus we (essentially) go from a Fourier mode of square magnitude $n$ to square magnitude $n+4$ in unit time. Thus wavenumber squared increases linearly, so the exponential rate of decay increases linearly, giving the $e^{-C^{-1} t^2}$ decay of Theorem~\ref{thm:4d}.

\subsubsection{Modification for Theorem~\ref{thm:4d}}

Finally let us discuss what we need to change for Theorem~\ref{thm:4d}. Since this result is stated with a qualitative rate, we can completely ignore the time scales involved: our only goal is to move mass from one pure Fourier mode to a larger pure Fourier mode while maintaining uniform $C^\infty$ regularity of the advecting flow. As discussed above, there are two ways to improve the regularity: 1) make $b$ smaller or 2) make $|\hat u(b)|$ smaller. Theorem~\ref{thm:3d} already maximally took advantage of making $|\hat u(b)|$ smaller, so we must instead focus on making $b$ smaller.

In order to make $b$ smaller, in fact unit sized throughout instead of $O(|a|)$, we add an additional dimension, working now in 4D. We take some $p \in \N$ sufficiently large, fixed throughout the argument. Then we build a flow that moves from Fourier mode $a = (m,n,\ell,p)$ to $ a+b = (m,n,\ell,-p-1)$ with $(m,n,\ell) \in \Z^3$ arbitrary. Thus we see that $b =(0,0,0, -2p-1)= O(1)$ throughout, thus this can be done with a uniform in $C^\infty$ regularity, regardless of $(m,n,\ell)$. This step however cannot be directly inducted to send our mass to infinite wavenumber. For that we need an additional step.

The considerations above that put a lower bound on how small $|\hat u(b)|$ can be while still moving the mass off of $z^0_t$ are only applicable in the case that $d_1 > d_0$. In the case that $d_1 < d_0$, we are in a rather different circumstance, where $z^1_t$ is where the mass ``wants to live''. As such, if $d_1 < d_0 \ll d_k$ for $ k \ne 0,1$, we do not even need to move all of the mass off of $z^0_t$. We can instead move \textit{some} mass onto $z^1_t$, and then wait sufficiently long. Since $d_1$ is the smallest diffusion coefficient, we will eventually be in the setting where we can apply the perturbative control of the second step, that is 
\[\frac{\sum_{k \ne 1} |z^k_t|^2}{\sum_k |z^k_t|^2} \ll 1.\]
The first push that moves some mass onto $z^1_t$ can be done with arbitrarily small magnitude, hence at arbitrary regularity, no matter how large $b$. The takeaway is the following: \textit{mass can be moved ``downhill'' at arbitrary regularity.}

Thus our second step is we move from the Fourier mode $(m,n,\ell,-p-1) \mapsto (x,y,z,p)$ where we again use Legendre's three square theorem so that $|(x,y,z,p)|^2 = |(m,n,\ell,p)| +4$, but $|(x,y,z,p)| < |(m,n,\ell,-p-1)|$, allowing this move to be done at arbitrary regularity. Combining the two steps, we have moved the mass to a Fourier mode with larger wavenumber using a velocity field that is uniformly bounded in $C^\infty.$

\subsection{Structure of the paper}

In the following Section~\ref{ss:proof of theorems from props}, we prove Theorems~\ref{thm:2d}--\ref{thm:4d} as a consequence of a sequence of propositions providing the existence of vector fields $u_t$ with the proper regularities that move mass from certain pure Fourier modes to other pure Fourier modes. It is then the goal of the remainder of the paper to prove these propositions.

Section~\ref{s:ODE} works purely with the infinite ODE system for $z^k_t$. Theorem~\ref{thm:ode-mass-moved} is the main result---phrased in the language of the ODE system $z^k_t$---which allows us to move mass ``uphill'', from a smaller wavenumber Fourier mode to a larger one. Proposition~\ref{prop:flow down hill} is the result that allows us to move mass ``downhill'' at arbitrary regularity. Theorem~\ref{thm:ode-mass-moved} is a consequence of Proposition~\ref{prop:move-all-mass-off-zero}---which codifies the first step of the sketch above in which all mass is moved off of $z^0_t$---and Corollary~\ref{cor:perturbative-control-iterated}---which codifies the second step of the sketch above in which the error is removed using the perturbative control argument. Proposition~\ref{prop:flow down hill} follows essentially similarly.

Section~\ref{s:proof of props} translates the results of Section~\ref{s:ODE} into results for the advection-diffusion equation instead of the $z^k_t$ ODE system. It then uses these translated results to prove the propositions of Section~\ref{ss:proof of theorems from props}, concluding the argument.

\subsection{Proof of Theorems~\ref{thm:2d}--\ref{thm:4d}}

\label{ss:proof of theorems from props}

The following proposition is the central ingredient to the proof of Theorem~\ref{thm:2d}.

\begin{proposition}
\label{prop:fourier to fourier 2d}
    There exists $r_0 \in \N, C,T>0$ such that for all $r \in \N$ with $r \geq r_0$, there exists a velocity field $u : [0,\infty) \times \T^2 \to \R^2$ with $\nabla \cdot u =0$ and $\|u\|_{L^\infty([0,\infty) \times \T^2)} \leq C$ such that if $\theta_0 = f_{(r,0)}$ and $\theta_t$ solves~\eqref{eq:advection-diffusion}, then there exists $\beta\in \C$ so that
    \[\theta_{T/r} = \beta f_{(0,r+1)}.\]
    Additionally, for all $t \in [0,T/r],$ we have that $\supp \hat \theta_t \subseteq \Z^2 - B_r.$ 
\end{proposition}

The following corollary is direct from applying the above proposition twice, the second time with the $x$ and $y$ coordinates swapped.

\begin{corollary}
    \label{cor:fourier to fourier 2d}
    There exists $r_0 \in \N, C,T>0$ such that for all $r \in \N$ with $r \geq r_0$, there exists a velocity field $u : [0,\infty) \times \T^2 \to \R^2$ with $\nabla \cdot u =0$ and $\|u\|_{L^\infty([0,\infty)\times\T^2)} \leq C$ such that if $\theta_0 = f_{(r,0)}$ and $\theta_t$ solves~\eqref{eq:advection-diffusion}, then there exists $\beta\in \C$ so that
    \[\theta_{T/r} = \beta f_{(r+2,0)}.\]
    Additionally, for all $t \in [0,T/r],$ we have that $\supp \hat \theta_t \subseteq \Z^2 - B_r.$ 
\end{corollary}

The following proposition is the central ingredient to the proof of Theorem~\ref{thm:3d}.

\begin{proposition}
\label{prop:fourier to fourier 3d}
    There exists $r_0 \in \N,C,T >0$ such that for all $(m,n,\ell) \in \Z^3$ such that $|(m,n,\ell)| \geq r_0$, there exists $(x,y,z) \in \Z^3$ with $|(m,n,\ell)|^2+1 \leq |(x,y,z)|^2 \leq |(m,n,\ell)|^2 +8$ and a velocity field $u: [0,\infty) \times \T^3 \to \R^3$ with $\nabla \cdot u = 0$ and $\|u\|_{L^\infty([0,\infty), W^{1,\infty}(\T^3))} \leq C$ such that if $\theta_0 = f_{(m,n,\ell)}$ and $\theta_t$ solves~\eqref{eq:advection-diffusion}, then there exists $\beta \in \C$ so that
    \[\theta_T = \beta f_{(x,y,z)}.\]
    Additionally, for all $t \in [0,T]$, we have that $\supp \hat \theta_t \subseteq \Z^3 - B_{|(m,n,\ell)|}.$
\end{proposition}

The following proposition is one of the central ingredients to the proof of Theorem~\ref{thm:4d}, providing the step that moves the Fourier mass ``uphill''.

\begin{proposition}
    \label{prop:fourier to fourier 4d-1}
    There exists $p \in \N,T>0$, such that $p \geq 10$ and for all $(m,n,\ell) \in \Z^3$ such that $(m,n,\ell) \ne 0$, there exists a velocity field $u : [0,\infty) \times \T^4 \to \R^4$ with $\nabla \cdot u =0$ and for all $s \in \N$ there exists $C(s)>0$ such that $\|u\|_{L^\infty([0,\infty), W^{s,\infty}(\T^4))} \leq C(s)$ and such that if $\theta_0 = f_{(m,n,\ell,p)}$ and $\theta_t$ solves~\eqref{eq:advection-diffusion}, then there exists $\beta \in \C$ so that
    \[\theta_T = \beta f_{(m,n,\ell,-p-1)}.\]
    Additionally, for all $t \in [0,T],$ we have that $\supp \hat \theta_t \subseteq \Z^4 - B_{|(m,n,\ell,p)|}$.
\end{proposition}

The following proposition is one of the central ingredients to the proof of Theorem~\ref{thm:4d}, providing the step that moves the Fourier mass ``downhill''.

\begin{proposition}
    \label{prop:fourier to fourier 4d-2}
    There exists $r_0 \in \N, T>0$ such that for all $(m,n,\ell) \in \Z^3$ with $|(m,n,\ell)| \geq r_0$, there exists $(x,y,z) \in \Z^3$ with $|(m,n,\ell)|^2 +1 \leq |(x,y,z)|^2 \leq |(m,n,\ell)|^2 +8$ and a velocity field $u : [0,\infty) \times \T^4 \to \R^4$ with $\nabla \cdot u =0$ and for all $s \in \N$, there exists $C(s)>0$ such that $\|u\|_{L^\infty([0,\infty), W^{s,\infty}(\T^4)} \leq C(s)$ and such that if $\theta_0 = f_{(m,n,\ell,-p-1)}$ with $p$ as in Proposition~\ref{prop:fourier to fourier 4d-1} and $\theta_t$ solves~\eqref{eq:advection-diffusion}, then there exists $\beta \in \C$ and $T>0$ so that
    \[\theta_T = \beta f_{(x,y,z,p)}.\]
    Additionally, for all $t \in [0,T]$, we have that $\supp \hat \theta_t \subseteq \Z^4 - B_{|(x,y,z,p)|}.$
\end{proposition}

The following corollary is direct from applying Proposition~\ref{prop:fourier to fourier 4d-1} and then Proposition~\ref{prop:fourier to fourier 4d-2}.
\begin{corollary}
\label{cor:fourier to fourier 4d}
    There exists $r_0, p \in \N$ such that for all $(m,n,\ell) \in \Z^3$ with $|(m,n,\ell)| \geq r_0$, there exists $(x,y,z) \in \Z^3$ with $|(m,n,\ell)|^2 + 1 \leq |(x,y,z)|^2 \leq |(m,n,\ell)|^2 +8$ and a velocity field $u : [0,\infty) \times \T^4 \to \R^4$ with $\nabla \cdot u =0$ and for all $s \in \N$, there exists $C(s)>0$ such that $\|u\|_{L^\infty([0,\infty), W^{s,\infty}(\T^4)} \leq C(s)$ and such that if $\theta_0 = f_{(m,n,\ell,p)}$ and $\theta_t$ solves~\eqref{eq:advection-diffusion}, then there exists $\beta \in \C$ and $T>0$ so that
    \[\theta_T = \beta f_{(x,y,z,p)}.\]
    Additionally, for $t \in [0,T]$, we have that $\supp \hat \theta_t \subseteq \Z^4 - B_{|(m,n,\ell,p)|}.$ 
\end{corollary}

The integrated form of the standard energy identity~\eqref{eq:energy-identity} will be used repeatedly below.

\begin{proposition}
\label{prop:intergrated-energy-identity}
    For $\theta_t$ a solution to~\eqref{eq:advection-diffusion}, we have that
 
 \begin{equation} \label{eq:energy-identity}\|\theta_t\|_{L^2} = \exp\Big({-}\int_0^t\frac{\|\nabla \theta_s\|_{L^2}^2}{\|\theta_s\|_{L^2}^2}\,ds \Big) \|\theta_0\|_{L^2}.
 \end{equation}
\end{proposition}
\begin{proof}
    Computing directly, we see that
    \[\frac{d}{dt} \|\theta_t\|_{L^2}^2 = -2\frac{\|\nabla \theta_t\|_{L^2}^2}{\|\theta_t\|_{L^2}^2} \|\theta_t\|_{L^2}^2.\]
    Integrating in time and taking a square root, we conclude.
\end{proof}

With the above results---the bulk of which will be proved in Section~\ref{s:proof of props}---stated, we are ready to prove Theorems~\ref{thm:2d}--\ref{thm:4d}.

\begin{proof}[Proof of Theorem~\ref{thm:2d}]
    Letting $r_0$ as in Corollary~\ref{cor:fourier to fourier 2d}, we take the initial data $\theta_0 := f_{(r_0,0)} \in C^\infty(\T^2)$ and $\theta_0 \ne 0$. We then iteratively apply Corollary~\ref{cor:fourier to fourier 2d} to construct a velocity field $u$ with $\nabla \cdot u =0$, $\|u\|_{L^\infty_{t,x}} \leq C$ and such that for the solution $\theta_t$ to~\eqref{eq:advection-diffusion} with initial data $\theta_0$, we have that for some universal $T>0$, and times $t_n$ defined by
    \[t_0 =0 \quad \text{and} \quad t_{n+1} = t_n + \frac{T}{r_0+2n},\]
    we have that
        \[\theta_{t_n} = \beta_n f_{(r_0+2n,0)},\]
    and for $t \in [t_n, t_{n+1}]$, $\supp \hat \theta_t \subseteq \Z^2 - B_{r_0 + 2n}.$ Thus for $t \in [t_n,t_{n+1}]$, we have that
    \[\frac{\|\nabla \theta_t\|_{L^2}^2}{\|\theta_t\|_{L^2}^2} \geq C^{-1} n^2.\]
    We note that for $n \geq 2$,
    \[C^{-1} \log n \leq t_n \leq C \log n.\]
    Putting the displays together, we have that for $t \geq C$, 
    \[\frac{\|\nabla \theta_t\|_{L^2}^2}{\|\theta_t\|_{L^2}^2} \geq C^{-1} e^{C^{-1} t}.\]
    Then by Proposition~\ref{prop:intergrated-energy-identity}, we have that for $t \geq C$
    \[\|\theta_t\|_{L^2} \leq \exp\Big({-}C^{-1}\int_C^t e^{C^{-1} t}\,dt\Big)\|\theta_0\|_{L^2} \leq \exp\big({-C^{-1}} e^{C^{-1} t}\big)\|\theta_0\|_{L^2}.\]
    We then conclude~\eqref{eq:2d bound}, choosing the prefactor constant large enough to cover times $t \leq C.$
\end{proof}

\begin{proof}[Proof of Theorem~\ref{thm:3d}]
    Letting $r_0$ as in Proposition~\ref{prop:fourier to fourier 3d}, we take initial data $\theta_0 := f_{(r_0,0,0)} \in C^\infty(\T^3)$ with $\theta_0 \ne 0.$ We then iteratively apply Proposition~\ref{prop:fourier to fourier 3d} to generate the velocity field $u$ on the interval $[nT, (n+1)T]$, with $T>0$ as in Proposition~\ref{prop:fourier to fourier 3d}, such that $u : [0,\infty) \times \T^3 \to \R^3$ with $\nabla \cdot u =0$, $\|u\|_{L^\infty_t W^{1,\infty}_x} \leq C$, and if $\theta_t$ solves~\eqref{eq:advection-diffusion} with initial data $\theta_0$, then for $(x_n,y_n,z_n) \in \Z^3$ with $(x_0, y_0,z_0) = (r_0,0,0)$, we have that for some $\beta_n \in \C$
    \begin{align*}
        |(x_n,y_n,z_n)|^2 &\geq r_0^2 + n,\\
        \theta_{nT} &= \beta_n f_{(x_n,y_n,z_n)},\\
        \text{for all } t \in [nT, (n+1)T],\; \supp \hat \theta_t &\subseteq \Z^3 - B_{|(x_n,y_n,z_n)|}.
    \end{align*}
    The applying Proposition~\ref{prop:intergrated-energy-identity}, we have that for $n \geq 2$, $ t\in [nT, (n+1)T),$
    \begin{align*}\|\theta_t\|_{L^2} &=  \exp\Big({-}\int_0^t\frac{\|\nabla \theta_s\|_{L^2}^2}{\|\theta_s\|_{L^2}^2}\,ds \Big) \|\theta_0\|_{L^2}
    \\&\leq  \exp\Big({-}C^{-1} \sum_{j=0}^{n-1}( r_0^2 +j )\Big) \|\theta_0\|_{L^2}
    \leq \exp\big({-C^{-1}} n^2) \leq e^{-C^{-1} t^2}.
    \end{align*}
    We then conclude Theorem~\ref{thm:3d}, choosing the prefactor constant large enough to cover the times $t \leq 2T$.
\end{proof}

\begin{proof}[Proof of Theorem~\ref{thm:4d}]
    Letting $r_0$ as in Corollary~\ref{cor:fourier to fourier 4d}, we take initial data $\theta_0 = f_{(r_0,0,0,p)} \in C^\infty(\T^4)$ with $\theta_0 \ne 0$. We then iteratively apply Corollary~\ref{cor:fourier to fourier 4d} to give a sequence of times $0 = T_0 < T_1 <\cdots$ with the velocity field $u : [0,\infty) \times \T^4 \to \R^4$ defined on each $[T_n,T_{n+1}]$ such that $\|u\|_{L^\infty_t W^{s,\infty}_x} \leq C(s)$ and so that there are sequences $(x_n,y_n,z_n) \in \Z^3$ with $(x_0,y_0,z_0) = (r_0,0,0)$ and $\beta_n \in \C$ such that
    \begin{align*}
        |(x_n,y_n,z_n)|^2 &\geq r_0^2 + n,\\
        \theta_{T_n} &= \beta_n f_{(x_n,y_n,z_n,p)},\\
        \text{for all } t \in [T_n, T_{n+1}], \supp \hat \theta_t &\subseteq \Z^4 - B_{|(x_n,y_n,z_n,p)|}.
    \end{align*}
    Combining this with Proposition~\ref{prop:intergrated-energy-identity} allows us to immediately conclude.
\end{proof}

\section{Control of the ODE system}

\label{s:ODE}

We consider solutions to the ODE system $z^k :
[0,\infty) \to \C$,
\begin{equation}
    \label{eq:main-ode}
   \dot z^k_t =  - d_k z^k_t  + i\sum_j v^j_t  z^{k-j}_t,
\end{equation}
with diffusion coefficients $d_k \geq 0$ and coefficient field $v^k : [0,\infty) \to \C$ subject to the constraint $\overline{v^{-k}} = v^k.$ We use the following notation throughout
\begin{align}
    \label{eq:M defn}
    M &:= \min_{k \in \Z, k \ne 0,1} d_k,\\
    \label{eq:S defn}
    S &:= \sum_{k \in \Z} \frac{1}{1+d_k}.
\end{align}

\begin{assumption}
    \label{asmp:main-ode}
    Throughout, we make the following assumptions on the diffusion coefficients $d_k \geq 0$:
    \begin{enumerate}
        \item \label{item:d0 d1} $d_0,d_1 \in [0,1]$.
        \item \label{item:M large} $M \geq 2^{26}$.
        \item \label{item:S small} $S \leq 6$.
        \item \label{item:dk separated} For all $k \in \N - \{0\}$, $d_{k+1} - d_{1-k} \geq 1.$  
    \end{enumerate}
\end{assumption}

The main goal of this section is to prove the following theorem, which, after translation to a statement about the advection-diffusion equation, will be the central tool for Section~\ref{s:proof of props}. This result says, given appropriate assumptions on $d_k$, we can build a coefficient field $v^k_t$ so that $z^k_t$ flows from initial data $\delta_{k,0}$ to a final value of $\beta \delta_{k,1}$, in the appropriate time scale and with the appropriate regularity (that is decay as $|k| \to\infty)$ in $v^k_t$. The proof will be given at the end of the section, using the results developed below.

\begin{theorem}
\label{thm:ode-mass-moved}
    There exists $T >0$ such that for any diffusion coefficients $d_k \geq 0$ satisfying Assumption~\ref{asmp:main-ode} and such that $d_0 =0 , d_1 =1$, there exists a coefficient field $v^k_t$ with $v^{-k}_t = \bar v^k_t$ and $v^0_t =0$ such that if $z^k_t$ is the solution to~\eqref{eq:main-ode} with coefficients $d_k, v^k_t$ and initial data $z^k_0 = \delta_{k,0}$, then for some $\beta \in \C$,
    \begin{equation}
    \label{eq:mass moved theorem}
    z^k_T = \beta \delta_{k,1}.
    \end{equation}
    Further, $v^k_t$ can be taken such that for all $n \in \N$, there exists $C(n)>0$ so that for all $k \in \N - \{0\},$
    \begin{equation}
    \label{eq:v regularity theorem}
    \sup_{t \geq 0} |v^k_t| \leq C (1 + d_{1-k})^{-n} e^{-M} + C (\delta_{k,1} + \delta_{-k,1}).
    \end{equation}
\end{theorem}

A slight modification of the argument for Theorem~\ref{thm:ode-mass-moved} will give the following proposition, which will also be used in Section~\ref{s:proof of props} in order to prove Proposition~\ref{prop:fourier to fourier 4d-2} in which the Fourier mass is moved ``downhill''. The proof of this proposition will also be given at the end of the section.

\begin{proposition}
\label{prop:flow down hill}
    For all $\eta>0$ and for any diffusion coefficients $d_k \geq 0$ satisfying Assumption~\ref{asmp:main-ode} and such that $d_1=0, d_0 =1$, there exists $T>0$ and a coefficient field $v^k_t$ with $v^{-k}_t= \bar v^k_t$ and $v^0_t =0$ such that if $z^k_t$ is the solution to~\eqref{eq:main-ode} with coefficients $d_k,v^k_t$ and initial data $z^k_0 = \delta_{k,0}$, then for some $\beta \in \C$,
    \begin{equation}
    \label{eq:down hill mass moved}
    z^k_T = \beta \delta_{k,1}.
    \end{equation}
    Further, $v^k_t$ can be taken such that for all $n \in \N$, there exists $C(n)> 0$ such that for all $k \in \N -\{0\}$,
    \begin{equation}
    \label{eq:down hill regularity}
    \sup_{t \geq 0} |v^k_t|  \leq C \eta (1+d_{1-k})^{-n}.
    \end{equation}
\end{proposition}

The following proposition codifies the first step described in Section~\ref{ss:overview}, in which all of the mass is moved off of $z^0_t.$ In order to be able to keep good track of the time scales, we want to also guarantee a controlled amount of mass is left on $z^1_t$. We note that we give a (semi-)explicit definition of the coefficient field $v_t^k$, however it is partially implicit as the coefficient field $v^k_t$ depends on values of $z^1_t, z^0_t.$ The well-definedness of the choice of coefficients $v^k_t$ in this implicit manner is straightforward when $z^0_t, z^1_t$ are away from $0$. We take the convention, which makes the well-definedness at all times also straightforward, that on $[\frac{1}{2^{10}}, \infty)$, $\frac{\bar z^0_t|z^1_t|}{\bar z^1_t|z^0_t|} := 0$---and hence $v^k_t =0$ for all $k$---for all times following the first time in $[\frac{1}{2^{10}}, \infty)$ that $z^0_t =0$ or $z^1_t=0$.

The explicit constants (e.g.\ $2^8,2^{10}$) are used in place of implicit constants so as to avoid subtle interdependency between the different implicit constants.

\begin{proposition}
    \label{prop:move-all-mass-off-zero}
    Let $d_k \geq 0$ be diffusion coefficients satisfying Assumption~\ref{asmp:main-ode} and such that $d_0=0, d_1=1$. We let $z^k_t$ be the solution to~\eqref{eq:main-ode} with initial data $z^k_0 = \delta_{k,0}$ and coefficient field $v^k_t$ defined by 
    \begin{align*}
        v^k_t &:= a_t\delta_{k,1} + \overline{a_t} \delta_{k,-1}\\
        a_t &= \begin{cases} 2^8 & t \in [0,\frac{1}{2^{10}})\\ - i 2^8  \frac{\bar z^0_t|z^1_t|}{\bar z^1_t|z^0_t|} & t \in [\frac{1}{2^{10}},\infty),\end{cases}
    \end{align*}
    where we take the convention that $\frac{\bar z^0_t|z^1_t|}{\bar z^1_t|z^0_t|} := 0$ in the case that $z^0_t =0$ or $z^1_t =0$.
    Then we have that
    \[z^0_{1} =0 \quad \text{and} \quad |z^1_1| \geq \frac{1}{96}.\]
\end{proposition}

\begin{proof}

   We note that
    \[\dot z^k_t =  - d_k z^k_t + i \big( a_t z^{k-1}_t + \overline{a_t} z^{k+1}_t\big).\]
    We first note that since $d_k \geq 0,$ we have that
    \[\frac{d}{dt} \sum_k |z^k_t|^2 = - 2 \sum_k d_k |z^k_t|^2 \leq 0,\]
    so that for all $t \geq 0,$
    \[ \sum_k |z^k_t|^2\leq \sum_k |z^k_0|^2 = 1.\]
    
    By Duhamel's principle, for any $t \geq 0,$
    \begin{equation}
    \label{eq:zk duhamel}
    z^k_t = e^{-d_kt}z^k_0 + i \int_0^t e^{-d_k(t-r)} \big( a_r z^{k-1}_r + \overline{a}_r z^{k+1}_r\big)\,dr.
    \end{equation}
    Thus have the estimates for $k \ne 0$, bounding $|a_t| \leq 2^8,$
    \begin{equation}
    \label{eq:zk estimate}
    |z^k_t| \leq  2^9\int_0^t e^{-d_k(t-r)}\,dr =\frac{2^9}{d_k} (1 - e^{-d_kt}) \leq \frac{2^9}{d_k} \land 2^9 t. 
    \end{equation}

    Using the equation for $z^0$, that $d_0 =0$, and that $z^0_0=1,$ this then implies
    \begin{equation}
    \label{eq:z0 close to 1}
    |z^0_t -1| =  \Big|\int_0^t \big(a_r z^{-1}_r + \bar a_r z^{1}_r\big)\,dr\Big| \leq 2^{18} \int_0^t r\,dr \leq 2^{17} t^2.
    \end{equation}
    Then using~\eqref{eq:zk duhamel} for $k=1$, that $d_1 = 1$, $z^1_0=0,$ and the explicit form $a_t$ for $t \in [0,2^{-10}),$ we have that
    \[z^1_{2^{-10}} - i 2^8 \int_0^{2^{-10}} e^{-(2^{-10}-r)} \,dr = i 2^8 \int_0^{2^{-10}} e^{-(2^{-10}-r)}  \big((z_r^0-1)+ z_r^2\big)\,dr,\]
    thus using~\eqref{eq:zk estimate} for $z^2$ and~\eqref{eq:z0 close to 1} to $(z^0_r -1)$, we see that
    \begin{align*}\Big|z^1_{2^{-10}} -  i 2^8\int_0^{2^{-10}} e^{-(2^{-10}-r)} \,dr\Big|  &\leq    2^8\int_0^{2^{-10}} e^{-(2^{-10}-r)}  \big(|z^0_r-1| +|z^2_r|\big)\,dr 
    \\&\leq 2^{17} \int_0^{2^{-10}} 2^8 r^2  + r\,dr
    \\&\leq \frac{1}{8}.
    \end{align*}
    Then computing directly, we see that
    \[2^8\int_0^{2^{-10}} e^{-(2^{-10}-r)} \,dr \geq \frac{1}{4} e^{-2^{-10}} \geq \frac{3}{16}.\]
    Thus
    \[|z^1_{2^{-10}}| \geq \frac{3}{16} - \frac{1}{8} \geq \frac{1}{16}.\]
    
     Then using the equations for $z^0, z^1$ and the definition of $a$, we have for $t \geq 2^{-10}$,
    \[\begin{cases} \dot z^0_t =-  2^8 \frac{z^0_t}{|z^0_t|} |z^1_t| +  i   a_t z^{-1}_t ,\\
\dot z^1_t = -  z^1_t+    2^8 \frac{ z^1_t}{|z^1_t|} |z^0_t| + i  \bar a_t z^2_t,
\end{cases}\]
    where we take $z^j_t/|z^j_t| =0$ in the case that $z^j_t =0$. Thus, using the global-in-time bounds of~\eqref{eq:zk estimate} on $z^{-1}, z^2,$ we have the differential inequalities on $[2^{-10},\infty),$
    \[\begin{cases}\frac{d}{dt} |z^0_t| \leq -   2^{8} |z^1_t| + \frac{2^{17} }{d_{-1}} \leq  -   2^8 |z^1_t| + 2^{-9},\\
\frac{d}{dt} |z^1_t| \geq - |z^1_t| - \frac{2^{17}}{d_2} \geq - |z^1_t| - 2^{-9},
\end{cases}\]
where we use Item~\ref{item:M large} of Assumption~\ref{asmp:main-ode}.
Then Gr\"onwall's inequality implies that for $0\leq s \leq 1 - 2^{-10}$,
    \begin{align*}|z^1_{2^{-10}}| &\leq e^{s} |z^1_{s+2^{-10}}| + 2^{-9} \int_{2^{-10}}^{s+2^{-10}} e^{(s+ 2^{-10}-r)}\,dr
    \\&\leq   e^{s} |z^1_{s+2^{-10}}| + 2^{-9} e^{s}
    \\&\leq 3 (|z^1_{s+2^{-10}}| + 2^{-9}),
    \end{align*}
    where we use that $s \leq 1$. Thus rearranging and using the lower bound on $z^1_{2^{-10}}$, we get that for $t \in [2^{-10}, 1],$
    \[|z_t^1| \geq \frac{1}{48} - 2^{-9} \geq \frac{1}{96}.\]
    We note this in particular implies the lower bound of $z^1_{1}$ in the proposition statement.
    
    Let $\tau \in (2^{-10},\infty]$ be the first time that $z^0_t = 0$. To conclude, it suffices to prove that $\tau \leq 1$, as once $z^0_t = 0$, the velocity field is defined to be $0$ from that point on, and so $z^0_t$ will remain zero from then on. We note that by the differential inequality for $|z^1_t|$, we have that 
    \begin{align*}1 &\geq -\int_{2^{-10}}^{1 \land \tau} \frac{d}{dr} |z^0_r|\,dr 
    \\&\geq \int_{2^{-10}}^{1 \land \tau}  2^8  |z^1_r| - 2^{-9}\,dr 
    \\&\geq 2^8 \big(1\land \tau - 2^{-10}\big) \Big( \frac{1}{96} - 2^{-17}\Big) 
    \\&\geq \frac{4}{3}  \big(1 \land \tau - 2^{-10}\big),
    \end{align*}
    Then we see though that we must have $\tau < 1$---allowing us to conclude---as otherwise the above inequality would be a contradiction. 
\end{proof}

The following proposition is the first part of the second step of the argument sketched in Section~\ref{ss:overview}; it is essentially one step of the Newton's method iteration. We assume that almost all of mass is on $z^1_0$ with a perturbative amount of mass on $z^k_t$ for $k \ne 1$. The goal is to construct a coefficient field $v^k_t$ that eliminates the error to leading order, leaving a lower order error behind. We perform this leading order correction on time intervals of arbitrarily small size, though pick up larger errors---as well as require smaller initial error---for smaller time intervals. Working on arbitrarily small time intervals is necessary to perform infinite iterations in a finite time horizon; we will end up choosing dyadic intervals for the iterations. The leading order error is corrected by solving a linear control problem. The linear control problem is strongly underdetermined---an infinite dimensional space of controls for a two dimensional configuration space. Instead of using a least squares method, we rather introduce an explicit basis for a two-dimensional subspace on which the problem is well-posed.

\begin{proposition}
    \label{prop:perturbative-control}
    Let $d_k \geq 0$ be diffusion coefficients satisfying Assumption~\ref{asmp:main-ode}. Let $\ep \in (0,1/2)$ and let $z^k_0$ be such that
    \[\sqrt{\frac{\sum_{k \ne 1} |z^k_0|^2}{\sum_k |z^k_0|^2}} \leq \ep.\]
    Then there exists $C>0$ such that for any $T \in (0,1],$ such that if
    \[\ep \leq C^{-1} T^{3/2},\]
    then defining for each $k \in \N - \{0\},$
            \begin{align*} B^k &= \frac{-i}{ e^{(d_{1-k} - d_1)T/2} - e^{(d_{k+1} - d_1)T/2}} \begin{pmatrix}
         e^{(d_{1-k} - d_1)T/2}   & -e^{(d_{k+1} - d_1)T/2}
        \\
        -1   & 1 
    \end{pmatrix},
    \\     \begin{pmatrix}a^k \\ b^k \end{pmatrix} &= B^k \begin{pmatrix}   -\Big(\int_0^{T/2} e^{(d_{k+1} -d_1)s}\,ds\Big)^{-1}z_0^{k+1} \\\Big(\int_0^{T/2} e^{(d_{1-k} -d_1)s}\,ds\Big)^{-1}\overline{z_0^{1-k}}\end{pmatrix},
    \\    v^k_t &= a^k \indc_{[0,T/2]}(t) + b^k \indc_{(T/2,T]}(t),\quad v^{-k} = \bar v^k,\quad \text{and}\quad v^0=0,
    \end{align*}
    then for the solution $z^k_t$ to~\eqref{eq:main-ode} with coefficients $d_k, v^k_t$, and initial data $z^k_0$, we have the bound 
    \begin{equation}
    \label{eq:main-perturbative-bound}
    \sqrt{\frac{\sum_{k \ne 1} |z^k_T|^2}{\sum_k |z^k_T|^2}} \leq C T^{-3} \ep^2.
    \end{equation}
    Additionally, for $k \in \N - \{0\}$ and $n \in \N$, we have that
    \begin{equation}
    \label{eq:v-regularity}
    \sup_{t \in [0,T]} |v^k_t| \leq  C e^{-d_{1-k}T/4}  T^{-2} (|z^{k+1}_0| + |z^{1-k}_0|) \leq C(n)  (1+d_{1-k})^{-n}T^{-2-n} \ep.
    \end{equation}
\end{proposition}

\begin{proof}
    Without loss of generality, by linearity, we can assume that $\sum_k |z^k_0|^2 = 1$ and $\sum_{k \ne 1} |z^k_0|^2 = \ep^2$.

    We decompose
    \[z^k_t = e^{-d_kt} z^k_0+ \gamma^k_t + \psi^k_t,\]
    where $\gamma^k_t$ is defined by
    \[\begin{cases}\dot \gamma^k_t = - d_k \gamma^k_t + i e^{-d_1 t} v^{k-1}_t\\
    \gamma^k_0 = 0.
    \end{cases}\]
    Thus the remainder term $\psi^k_t$ solves
    \[\begin{cases}\dot \psi^k_t = - d_k \psi^k_t + i \sum_{j \ne1} e^{-d_j t}v_t^{k-j}z_0^j + i \sum_{j} v^{k-j}_t \gamma^j_t + i \sum_j v^{k-j}_t \psi^j_t \\ \psi^k_0 =0.\end{cases}\]
    By Duhamel's principle, we have that
    \begin{equation}
    \label{eq:gamma-duhamel}
    \gamma^k_t = i \int_0^t e^{-d_k(t-s)} e^{-d_1 s} v^{k-1}_s\,ds.
    \end{equation}
    We claim that for $k \in \Z, k\ne 1$.
    \begin{equation}
    \label{eq:gamma claim}
    \gamma^k_T = -e^{-d_kT} z_0^k,
    \end{equation}
    that is our choice of $v^k_t$ is such that $\gamma^k_T$ exactly cancels the leading order error. Fix $k \in \N - \{0\}$, and then using the definition of $v^k_t$, we have that
    \begin{align*}\gamma^{k+1}_T &= i\int_0^{T/2} e^{- d_{k+1} (T-s)} e^{-d_1 s} a^k\,ds + i\int_{T/2}^{T} e^{-d_{k+1} (T-s)} e^{-d_1 s} b^k\,ds
    \\&= i e^{-d_{k+1} T} \Big(\int_0^{T/2} e^{(d_{k+1} - d_1)s}\,ds\Big) \big(a^k + e^{(d_{k+1} - d_1)T/2} b^k\big).
    \end{align*}
    Similarly,
    \[\gamma^{1-k}_T = i e^{-d_{1-k} T} \Big(\int_0^{T/2} e^{(d_{1-k} - d_1)s}\,ds\Big) \big(\bar a^k + e^{(d_{1-k} - d_1)T/2} \bar b^k\big).\]
    Rearranging into a matrix equation, we see
        \[-i\begin{pmatrix}
        1    & e^{(d_{k+1} - d_1)T/2}
        \\
        1   & e^{(d_{1-k} - d_1)T/2}
    \end{pmatrix} \begin{pmatrix}
        a^k \\ b^k
    \end{pmatrix}= \begin{pmatrix} -\Big(\int_0^{T/2} e^{(d_{k+1} -d_1)s}\,ds\Big)^{-1}e^{d_{k+1}T}\gamma^{k+1}_T \\ \Big(\int_0^{T/2} e^{(d_{1-k} -d_1)s}\,ds\Big)^{-1} e^{d_{1-k}T}  \overline{\gamma^{1-k}_T} \end{pmatrix}.\]
    We note then the matrix multiplying $\begin{pmatrix} a^k \\ b^k\end{pmatrix}$ is exactly $-(B^k)^{-1}$. Thus inserting the definition of $\begin{pmatrix} a^k \\ b^k\end{pmatrix}$ and cancelling $(B^k)^{-1} B^k =I$, we see 
    \[\begin{pmatrix}   \Big(\int_0^{T/2} e^{(d_{k+1} -d_1)s}\,ds\Big)^{-1}z_0^{k+1} \\-\Big(\int_0^{T/2} e^{(d_{1-k} -d_1)s}\,ds\Big)^{-1}\overline{z_0^{1-k}}\end{pmatrix} = \begin{pmatrix} -\Big(\int_0^{T/2} e^{(d_{k+1} -d_1)s}\,ds\Big)^{-1}e^{d_{k+1}T}\gamma^{k+1}_T \\ \Big(\int_0^{T/2} e^{(d_{1-k} -d_1)s}\,ds\Big)^{-1} e^{d_{1-k}T}  \overline{\gamma^{1-k}_T} \end{pmatrix}.\]
    Rearranging, we see that as claimed 
    \[\gamma_T^{k+1} = -e^{-d_{k+1} T} z_0^{k+1} \quad \text{and} \quad \gamma_T^{1-k} = -e^{-d_{1-k}T} z_0^{1-k}.\]
    Since $k \in \N -\{0\}$ is arbitrary, we see we get the claim~\eqref{eq:gamma claim}. Note also that since $v^0_t =0$, $\gamma^1_t =0.$ Thus
    \begin{equation}
    \label{eq:gamma-gone}
    z_T^k = e^{-d_1 T} z_0^1 + \psi^k_T.
    \end{equation}
    
    Our goal now is to estimate $\|\psi^k_T\|_{\ell^2}.$ The first step is to get control on $a^k, b^k$. To that end, we note that
    \begin{align*}
        \frac{|a^k| + |b^k|}{|z^{k+1}_0| + |z^{1-k}_0|} &\leq C \frac{e^{(d_{k+1} -d_1)T/2}  \Big(\int_0^{T/2} e^{(d_{1-k} -d_1)s}\,ds\Big)^{-1}}{  e^{(d_{k+1} - d_1)T/2}-e^{(d_{1-k} - d_1)T/2}}
        \\&\leq C e^{-(d_{1-k} -d_1)T/4}  T^{-1} \big(1 - e^{-(d_{k+1} - d_{1-k})T/2}\big)^{-1}. 
    \end{align*}
    Then 
      \[1 - e^{-(d_{k+1} - d_{1-k})T/2} \geq (d_{k+1} - d_{1-k})T/4 \land \frac{1}{2},\]
    so putting the displays together, we have that 
    \begin{align*}
    |a^k| + |b^k| &\leq C e^{-(d_{1-k} -d_1)T/4}  T^{-1} ((d_{k+1} - d_{1-k})^{-1} T^{-1} + 1)(|z^{k+1}_0| + |z^{1-k}_0|)
    \\&\leq C e^{-d_{1-k}T/4}  T^{-2} (|z^{k+1}_0| + |z^{1-k}_0|),
    \end{align*}
    where we use that $d_{k+1} - d_{1-k} \geq 1$ by Item~\ref{item:dk separated} of Assumption~\ref{asmp:main-ode}, as well as that $T,d_1 \leq 1$. We note that this directly implies~\eqref{eq:v-regularity}. Then we see by the definition of $v^k_t$ that for all $t \in [0,T]$,
    \begin{equation}
    \label{eq:vt-ell-2-bound}
    \|v_t\|_{\ell^2} \leq CT^{-2} \Big(\sum_{k \ne 1} |z_0^k|^2\Big)^{1/2} \leq CT^{-2} \ep.\end{equation}
    Then we see that by~\eqref{eq:gamma-duhamel}, for $t \in [0,T]$
    \begin{equation}
    \label{eq:gammat-ell-2-bound}
    \|\gamma_t\|_{\ell^2} \leq T \sup_{s \in [0,T]}\|v_s\|_{\ell^2} \leq C T^{-1} \ep.
    \end{equation}
    Now, using the equation for $\psi_t$, using that $\overline{v_k} = v_{-k}$ to show the final term is $0$, we have that
    \begin{align*}\frac{d}{dt} \frac{1}{2}  \|\psi_t\|_{\ell^2}^2 &\leq - \sum_k d_k |\psi^k_t|^2 + \sum_{j,k} v^{k-j}_t (\indc_{j \ne 0} z^j_0) \psi^k_t + \sum_{k,j} |v_t^{k-j} |\gamma^j_t| |\psi^k_t|
    \\&\leq - \sum_k d_k |\psi^k_t|^2 + \|v_t\|_{\ell^2} \|\indc_{j \ne 0} z_0\|_{\ell^2} \|\psi_t\|_{\ell^1} + \|v_t\|_{\ell^2} \|\gamma_t\|_{\ell^2} \|\psi_t\|_{\ell^1}, 
    \end{align*}
    where we use Young's convolution inequality. Then we note that
    \[\|\psi_t\|_{\ell^1} \leq \|(1+d_k)^{1/2} \psi_t\|_{\ell^2} \Big(\sum_k \frac{1}{1+d_k}\Big)^{1/2} \leq C \|(1+d_k)^{1/2} \psi_t\|_{\ell^2},\]
    where we use Item~\ref{item:S small} of Assumption~\ref{asmp:main-ode}. Then, applying Young's inequality, we have
    \[\frac{d}{dt} \|\psi_t\|_{\ell^2}^2 \leq \|\psi_t\|_{\ell^2}^2 + C \|v_t\|_{\ell^2}^2 \|\indc_{j \ne 0} z_0\|_{\ell^2}^2 +\|v_t\|_{\ell^2}^2 \|\gamma_t\|_{\ell^2}^2 \leq \|\psi_t\|_{\ell^2}^2 + C T^{-4} \ep^4 +CT^{-6} \ep^4,\]
    using the bounds~\eqref{eq:vt-ell-2-bound} and~\eqref{eq:gammat-ell-2-bound}. Thus by Gr\"onwall's inequality and using that $T \leq 1$,
    \[\|\psi_T\|_{\ell^2} \leq C T^{-3} \ep^2.\]
    Then combining this with~\eqref{eq:gamma-gone}, we have that
    \[\sqrt{\frac{\sum_{k \ne 1} |z^k_T|^2}{\sum_k |z^k_T|^2}} \leq \frac{\|\psi_T\|_{\ell^2}}{e^{-d_1 T} z_0^1 - \|\psi_T\|_{\ell^2}} \leq C T^{-3} \ep^2,\]
    where we use that since $\ep \leq \frac{1}{2},$ $z^1_0 \geq \frac{1}{2}$. Then $d_1, T \leq 1$, so $e^{-d_1 T} z_0^1 \geq \frac{1}{8}$, and finally we use that by assumption $\|\psi_T\|_{\ell^2} \leq C T^{-3} \ep^2 \leq \frac{1}{16},$ which together gives that $e^{-d_1 T} z_0^1 - \|\psi_T\|_{\ell^2} \geq \frac{1}{16}$. We note this is directly~\eqref{eq:main-perturbative-bound}.
\end{proof}

We now iterate Proposition~\ref{prop:perturbative-control} on dyadic intervals to force the error to $0$ on $[0,1]$.

\begin{corollary}
        \label{cor:perturbative-control-iterated}
      Let $d_k \geq 0$ be diffusion coefficients satisfying Assumption~\ref{asmp:main-ode}. Let $\ep \in (0,1/2)$ and let $z^k_0$ be such that
    \[\sqrt{\frac{\sum_{k \ne 1} |z^k_0|^2}{\sum_k |z^k_0|^2}} \leq \ep.\]
    Then there exists $C>0$ such that if 
    \[\ep \leq C^{-1},\]
    then there exists a coefficient field $v^k_t : [0,\infty) \to \C$ with $v^{-k} = \bar v^k, v^0 =0$ such that for some $\beta \in \C$,
    \begin{equation}
    \label{eq:mass-moved}
    z^k_1 = \beta \delta_{k,1}.
    \end{equation}
    Additionally, $v^k$ can be taken so that for each $k \in \N - \{0\}$ and each $n \in \N$, 
    \begin{equation}
    \label{eq:v-regularity-iterated}
    \sup_{t \in [0,1]} |v^k_t| \leq C(n) (1+d_{1-k})^{-n}\ep^{1/2}.
    \end{equation}
\end{corollary}

\begin{proof}
    We want to inductively apply Proposition~\ref{prop:perturbative-control} with $T_j = 2^{-j}$ for $j =1,2,...$. Define for $j=0,1,2,...$,
    \[\ep_j := \sqrt{\frac{\sum_{k \ne 1} |z^k_{1 - 2^{-j}}|^2}{\sum_{k} |z^k_{1 - 2^{-j}}|^2}},\]
    so that $\ep_0 =\ep.$ In order to apply Proposition~\ref{prop:perturbative-control}, we need to verify at each step that
    \[\ep_j \leq D^{-1} 2^{-3j/2}\]
    for some $D>0$ coming from Proposition~\ref{prop:perturbative-control}. Provided this condition holds, Proposition~\ref{prop:perturbative-control} gives that
    \[\ep_{j+1} \leq D 2^{3j} \ep_j^2.\]
    As such, we define the recursive sequence $\delta_0 = \ep,$ $\delta_{j+1} = D 2^{3j} \delta_j^2$. In order to conclude, we just need to show that
    \begin{equation}
    \label{eq:delta-bound}
    \delta_j \leq D^{-1} 2^{-3j/2},
    \end{equation}
    as if that holds, we inductively get that $\ep_j \leq \delta_j \to 0.$

    We compute $\delta_j$ explicitly:
    \begin{align*}
    \delta_j &= D 2^{3j} \delta_{j-1}^2 
    \\&= D 2^{3j} \big( D 2^{3(j-1)}\delta_{j-2}^2\big)^2 
    \\&= D^{2^0 + 2^1} 2^{3j2^0 + 3(j-1)2^1} \big(D 2^{2(j-2)}\delta_{j-3}^2\big)^{2^2}
    \\&=\cdots = D^{\sum_{k=1}^j 2^{j-k}} 2^{ 3\sum_{k=1}^j k 2^{j-k}} \ep^{2^j}
    \\&\leq (D2^{3 \sum_{k=1}^j k 2^{-k}} \ep)^{2^j} \
    \\&\leq (C\ep)^{2^j} .
    \end{align*}
    Thus choosing $\ep>0$ sufficiently small depending $D$, we see we get~\eqref{eq:delta-bound} as well as
    \[\delta_j \leq 2^{-2^j} \ep^{1/2},\] 
    allowing us in particular to conclude~\eqref{eq:mass-moved}.
    
    For~\eqref{eq:v-regularity-iterated}, by~\eqref{eq:v-regularity}, we have that
    \begin{align*}\sup_{t \in [0,1]} |v^k_t| &\leq C(n)  d_{1-k}^{-n}\sup_{j \in \N} 2^{(2+n)j} \ep_j 
    \\&\leq C(n)  (1+d_{1-k})^{-n}\sup_{j \in \N} 2^{(2+n)j}  2^{-2^j} \ep^{1/2}
    \\&\leq C(n) (1+d_{1-k})^{-n} \ep^{1/2},
    \end{align*}
    as desired.
\end{proof}

By combining Proposition~\ref{prop:move-all-mass-off-zero} with Corollary~\ref{cor:perturbative-control-iterated}, we are now ready to prove Theorem~\ref{thm:ode-mass-moved}.

\begin{proof}[Proof of Theorem~\ref{thm:ode-mass-moved}]
    We define a coefficient field $v^k_t$. On $t \in [0,1]$, we take $v^k_t$ as in Proposition~\ref{prop:move-all-mass-off-zero}. We fix $\tau \geq 1$ and on $t \in [1,1+\tau]$, we take $v^k_t = 0$ for all $k \in \Z$. On $[1+\tau,2+\tau]$, we take $v^k_t$ as in Corollary~\ref{cor:perturbative-control-iterated}. Finally, we take $v^k_t = 0$ for $t > 2+\tau$. We first seek to prove~\eqref{eq:mass moved theorem}. We note this is direct from Corollary~\ref{cor:perturbative-control-iterated}, provided we can verify the hypotheses.

    First however, we note that the hypotheses of Theorem~\ref{thm:ode-mass-moved} are precisely those of Proposition~\ref{prop:move-all-mass-off-zero}, as such at time $1$, we have that
    \begin{equation}
    \label{eq:time 1 z}
    z^0_1 =0 \quad \text{and} \quad |z^1_1| \geq \frac{1}{96},
    \end{equation}
    and further
    \begin{equation}
    \label{eq:small time regularity bound}
    \sup_{t \in [0,1+\tau)} |v^k_t| \leq C \big(\delta_{k,1} + \delta_{k,-1}\big).
    \end{equation}
    Using that $v^k_t =0$ on $(1,1+\tau)$, using that $d_1 =1$ and~\eqref{eq:time 1 z}, we have that
    \[z^0_{1+\tau} =0 \quad \text{and} \quad |z^1_{1+\tau}| \geq \frac{e^{-\tau}}{96},\]
    while for $k \ne 0,1,$
    \[|z^k_{1+\tau}| \leq e^{-M\tau} |z^k_1|.\]
    Thus 
    \[\sum_k |z^k_0|^2 \geq \frac{e^{-2\tau}}{96^2},\]
    while
    \[\sum_{k \ne 1} |z^k_0|^2 \leq e^{-2M\tau} \sum_{k \ne 1} |z^k_1|^2 \leq e^{-2M\tau}.\]
    Therefore,
    \[\sqrt{\frac{\sum_{k \ne 1} |z^k_0|^2}{\sum_k |z^k_0|^2}}  \leq 96 e^{-(M-1)\tau} \leq 96  e^{-M\tau/2}.\]
    Choosing then $\tau \geq 4$ sufficiently large, we verify the hypotheses of Corollary~\ref{cor:perturbative-control-iterated}, yielding~\eqref{eq:mass moved theorem} by~\eqref{eq:mass-moved} and we get by~\eqref{eq:v-regularity-iterated} that for any $k \in \N - \{0\}, n \in \N,$
    \[\sup_{t \in [\tau+1, \infty)} |z^k_t| \leq C(n) (1+d_{1-k})^{-n} e^{-M}.\]
    Combining the above with~\eqref{eq:small time regularity bound} gives~\eqref{eq:v regularity theorem}.
\end{proof}

\begin{proof}[Proof of Proposition~\ref{prop:flow down hill}]
    We note that $T>0$ can depend on the diffusion coefficients, so a ``soft argument'' suffices: we don't need to keep track of how large $T$ is provided it stays finite.

    For some $\tau>0$ to be determined, we define for $t \in [0,\tau)$
    \[v^k_t = \eta \delta_{k,1} + \eta \delta_{k,-1}.\]
    We note then that directly from~\eqref{eq:main-ode}, using that $z^1_0 = z^{2}_0 = 0$ and $z^0_0=1$.
    \[\dot z^1_0 = i \eta.\]
    Thus (e.g.\ by Taylor expansion) there exists a choice of $\tau$ sufficiently small so that
    \[\rho := |z^1_\tau|>0.\]
    Then for some $\sigma>0$, we define on $[\tau, \tau+\sigma)$, $v^k_t =0.$ We thus have for $k \in \Z$,
    \[|z^k_{\tau+\sigma}| \leq e^{-d_k\sigma} |z^k_\tau|.\]
    Thus using that $d_1 =0$ and for $k \ne 1$ $d_k \geq 1$, we have that
    \[\sqrt{\frac{\sum_{k \ne 1} |z^k_{\sigma+\tau}|^2}{\sum_k |z^k_{\sigma + \tau}|^2}}  \leq e^{-\sigma} \frac{\sqrt{\sum_{k} |z^k_\tau|^2}}{\rho} \leq \rho^{-1} e^{-\sigma}.\]
    Taking then $\sigma$ sufficiently large, we verify the hypotheses of Corollary~\ref{cor:perturbative-control-iterated}. Then~\eqref{eq:down hill mass moved} is direct from~\eqref{eq:mass-moved} and~\eqref{eq:v-regularity-iterated} gives that
    \[\sup_{t \in [\tau+\sigma, \tau+\sigma+1]} \leq C(n) (1+d_{1-k})^{-n} \rho^{-1/2} e^{-\sigma/2}.\]
    Choosing $\sigma$ sufficiently large and using that $|v^k_t| \leq \eta$ for $t \leq \tau +\sigma$, we conclude~\eqref{eq:down hill regularity}.
\end{proof}

\section{Proof of Propositions~\ref{prop:fourier to fourier 2d}--\ref{prop:fourier to fourier 4d-2}}
\label{s:proof of props}

In Section~\ref{s:ODE}, we have completed the technical crux of the proof. What remains is to translate the ODE results into results about the advection-diffusion equation and then to apply those translated results to prove Propositions~\ref{prop:fourier to fourier 2d}--\ref{prop:fourier to fourier 4d-2}.

\subsection{Translating from ODE to PDE}

We use the notation that we are starting from the Fourier mode $a \in \Z^d - \{0\}$ and want to move the mass onto the Fourier mode $a+b$ for some $b \in \Z^d - \{0\}.$ The following definition and proposition translate between the ODE system and the advection-diffusion equation. Note that even after fixing $a,b$, we have the free parameters $A,L$. These will be used to make $d_0 = 0, d_1=1$ or $d_0=1, d_1=0$ so as to fit into the hypotheses of Theorem~\ref{thm:ode-mass-moved} or Proposition~\ref{prop:flow down hill} respectively.

\begin{definition}
\label{defn:1d-to-any-d}
    For any $a,b \in \Z^d- \{0\}$ with $|a \cdot b| < |a| |b|$, define 
    \[\ell_{a,b} = (|a|\alpha_{a,b})^{-1} \Big(a - \frac{a \cdot b}{|b|^2} b\Big) \quad \text{and}\quad \alpha_{a,b} := \Big(1 - \frac{(a\cdot b)^2}{|a|^2|b|^2}\Big)^{1/2},\]
    noting that $|\ell_{a,b}| = 1$, $\ell_{a,b} \cdot b =0.$ If we also have $v : [0,\infty) \times \T \to \R$, define for any $L>0$, $w_{a,b,v,L} : [0,\infty) \times \T^d \to \R^d$ by
    \[w_{a,b,v,L}(t,x) :=  \frac{L\ell_{a,b}}{\alpha_{a,b}|a|} v(Lt, b \cdot x).\]
\end{definition}

\begin{proposition}
\label{prop:ode-to-pde}
    For any $a,b \in \Z^d- \{0\}$ with $|a \cdot b| < |a| |b|$ and any $v : [0,\infty) \times \T \to \R$, taking $\alpha_{a,b}, \ell_{a,b},$ and $w_{a,b,v,L}$ as defined in Definition~\ref{defn:1d-to-any-d}, we have the following properties.
    \begin{enumerate}
        \item \label{item:div-free} $\nabla \cdot w_{a,b,v,L} = 0.$
        \item \label{item:regularity} $\|w_{a,b,v,L}\|_{L^\infty([0,\infty), W^{n,\infty}(\T^d))} \leq \frac{|b|^n L}{\alpha_{a,b}|a|} \|v\|_{L^\infty([0,\infty), W^{n,\infty}(\T^d))}.$
        \item \label{item:equation} Let $\theta_t$ solve~\eqref{eq:advection-diffusion} with advecting flow $w_{a,b,v,L}$ and initial data $\theta_0 = f_a$. Let $z^k_t$ solve~\eqref{eq:main-ode} with coefficient field $v^k_t := \hat v(t,k)$, dissipation coefficients
        \[d_k := \frac{|a+kb|^2}{L} - A,\]
        and initial data $z^k_0 = \delta_{k,0}$. Then for all $t \geq 0$,
        \[\theta_{t/L} = e^{-At}\sum_{k \in \Z} z^k_t f_{a + kb}.\]
            \end{enumerate}
\end{proposition}
\begin{proof}
    Item~\ref{item:div-free} is direct from $b \cdot \ell_{a,b} = 0.$ Item~\ref{item:regularity} is also a direct computation. For Item~\ref{item:equation}, taking the time derivative of both sides gives
    \begin{align*}
        \sum_{k \in \Z} (\dot z^k_t  - A z^k_t)f_{a+kb} &= \frac{1}{L} \dot \theta_{t/L}
        \\&= \frac{1}{L} \Delta \theta(t/L,x) + \frac{1}{L} \nabla \cdot \big(w_{a,b,\ell}(t/L,x) \theta(t/L,x)\big)
        \\&=  -\sum_{k \in \Z} \frac{|a+kb|^2}{L} z^k_t f_{a+kb} + \frac{1}{\alpha_{a,b}|a|}\sum_{m,n \in \Z} \nabla \cdot ( \ell_{a,b} v^m_{t}z^n_t  f_{mb} f_{a+nb})
        \\&=  -\sum_{k \in \Z} \frac{|a+kb|^2}{L} z^k_t f_{a+kb} + \frac{i}{\alpha_{a,b}|a|}\sum_{m,n \in \Z}  ((m+n)b + a)\cdot  \ell_{a,b} v^m_{t}z^n_tf_{a+(m+n)b})
        \\&=   -\sum_{k \in \Z} \frac{|a+kb|^2}{L} z^k_t f_{a+kb} + i\sum_{k,j \in \Z}v^j_{t}z^{k-j}_tf_{a+kb},
    \end{align*}
    agreeing with the equation for $z^n_t$, as claimed.
\end{proof}

We now can write a version of Theorem~\ref{thm:ode-mass-moved} but for the advection-diffusion equation. This somewhat messy statement allows us to move mass from the Fourier mode $a$ to $a+b$ given some computational verifiable hypotheses. It further records the time needed to move the mass as well as the regularity of the advecting flow.

\begin{corollary}
\label{cor:move mass theta}
    Let $a,b \in \Z^d - \{0\}$ with $|a \cdot b| < |a||b|$ such that for all $k \in \Z$, $|a + kb| \geq |a|$. Then define for $k \in \Z$,
    \[d_k := \frac{|a+kb|^2}{L} - A,\quad A := \frac{|a|^2}{L},\quad L:= |a+b|^2- |a|^2,\]
    so that $d_0=0, d_1=1.$ Defining, as above,
    \[M:= \min_{k\in \Z, k \ne 0,1} d_k \quad\text{and}\quad S:= \sum_{k \in \Z} \frac{1}{1+d_k},\]
    suppose that
    \begin{enumerate}
        \item \label{item:M large theta} $M \geq 2^{26},$
        \item \label{item:S small theta} $S \leq 6,$
        \item \label{item:d spacing theta} For all $k \in \N - \{0\}, d_{k+1} - d_{1-k} \geq 1.$
    \end{enumerate}
    Then there exists a universal $T>0$ and a velocity field $u : [0,\infty) \times \T^d \to \R^d$ depending on $a,b$ such that if we let $\theta_t$ solve~\eqref{eq:advection-diffusion} with advecting flow $u$ and initial data $\theta_0 = f_a$, then 
    \begin{enumerate}
        \item \label{item:u div-free} $\nabla \cdot u =0$,
        \item \label{item:u regularity} for all $n \in \N$, 
        \[\|u\|_{L^\infty([0,\infty), W^{n,\infty}(\T^d))} \leq  \frac{C(n) |b|^n (|a+b|^2 - |a|^2)}{|a| \alpha_{a,b}}\Big(1+ e^{-M}\sum_{k= 0}^\infty \frac{|k|^n}{(1 + d_{-k})^{n+1}}\Big),\]
         \item \label{item:theta fourier support} for all $t \geq 0$, $\supp \hat \theta_t \subseteq \Z^2 - B_{|a|}$,
        \item \label{item:theta mass moved} for some $\beta \in \C$, $\theta_{T/L} = \beta f_{a+b}.$
    \end{enumerate}
\end{corollary}

\begin{proof}
    We note that our assumptions directly allow us to apply Theorem~\ref{thm:ode-mass-moved} to give a universal $T>0$ and a coefficient field $v^k_t$ with $v^{-k}_t = \bar v^k_t$ such that for $z^k_t$ the solution to~\eqref{eq:main-ode} with coefficients $d_k, v^k_t$ and initial data $z^k_0 = \delta_{k,0}$, such that for some $\beta \in \C$, 
    \begin{equation}
    \label{eq:z mass moved}
    z^k_T = e^{AT}\beta \delta_{k,1},
    \end{equation}
    and for all $n \in \N$, there exists $C(n) >0$ such that for all $k \in \N - \{0\}$,
    \begin{equation}
    \label{eq:vk pointwise}
    \sup_{t \geq 0} |v^k_t| \leq C (1 + d_{1-k})^{-n} e^{-M/4} + C (\delta_{k,1} + \delta_{-k,1}).
    \end{equation}
    We let $v : [0,\infty) \times \T \to \R$ be given by $\hat v_t(k)= v^k_t$, noting that $v$ is $\R$-valued as $v^{-k}_t = \bar v^k_t.$ We then let $u = w_{a,b,v,L}$ as in Definition~\ref{defn:1d-to-any-d}. Then by Proposition~\ref{prop:ode-to-pde}, we have that for $t \geq 0,$
    \begin{equation}
    \label{eq:theta-rep}
    \theta_{t/L} = e^{-At} \sum_{k \in \Z} z^k_t f_{a+kb}.
    \end{equation}
    We now verify our conclusions. Item~\ref{item:u div-free} is direct from Item~\ref{item:div-free} of Proposition~\ref{prop:ode-to-pde}. Item~\ref{item:theta fourier support} is direct from the representation given by~\eqref{eq:theta-rep} and the hypothesis that $|a+kb| \geq |a|$ for $k \in \Z$. Then Item~\ref{item:theta mass moved} is direct from~\eqref{eq:theta-rep} and~\eqref{eq:z mass moved}. Thus all that remains is to verify Item~\ref{item:u regularity}. To that end, we note that by Item~\ref{item:regularity} of Proposition~\ref{prop:ode-to-pde} and~\eqref{eq:vk pointwise},
    \begin{align*}
        \|u\|_{L^\infty_t W^{n,\infty}_x} &\leq\frac{|b|^n L}{\alpha_{a,b} |a|} \|v\|_{L^\infty([0,\infty), W^{n,\infty})}
        \\&\leq \sup_{t \geq 0}\frac{C(n) |b|^n (|a+b|^2 - |a|^2)}{\alpha_{a,b}|a|} \sum_{k \in \Z} |k|^n |v^k_t|
        \\&\leq\frac{C(n) |b|^n (|a+b|^2 - |a|^2)}{\alpha_{a,b}|a|}\Big(1+ e^{-M}\sum_{k= 2}^\infty \frac{|k|^n}{(1 + d_{1-k})^{n+1}}\Big)
        \\&\leq \frac{C(n) |b|^n (|a+b|^2 - |a|^2)}{\alpha_{a,b}|a|}\Big(1+ e^{-M}\sum_{k= 0}^\infty \frac{|k|^n}{(1 + d_{-k})^{n+1}}\Big),
    \end{align*}
    as claimed.
\end{proof}

We now write a version of Proposition~\ref{prop:flow down hill} but for the advection-diffusion equation, just as we did for Theorem~\ref{thm:ode-mass-moved}.

\begin{corollary}
    \label{cor:move mass theta downhill}
    Let $a,b \in \Z^d - \{0\}$ with $|a \cdot b| \leq |a||b|$ such that for all $k \in \Z$, $|a+kb| \geq |a+b|$. Then define for $k \in \Z$, 
    \[d_k := \frac{|a+kb|^2}{L} - A,\quad A:= \frac{|a+b|^2}{L}, \quad L := |a|^2 - |a+b|^2,\]
    so that $d_0 = 1, d_1 =0$.  Defining, as above,
    \[M:= \min_{k\in \Z, k \ne 0,1} d_k \quad\text{and}\quad S:= \sum_{k \in \Z} \frac{1}{1+d_k},\]
    suppose that
    \begin{enumerate}
        \item \label{item:M large theta downhill} $M \geq 2^{26},$
        \item \label{item:S small theta downhill} $S \leq 6,$
        \item \label{item:d spacing theta downhill} For all $k \in \N - \{0\}, d_{k+1} - d_{1-k} \geq 1.$
    \end{enumerate}
    Then for any $\eta>0$ there exists a velocity field $u : [0,\infty) \times \T^d \to \R^d$ and a time $T>0$ depending on $a,b,\eta$ such that if we let $\theta_t$ solve~\eqref{eq:advection-diffusion} with advecting flow $u$ and initial data $\theta_0 = f_a$, then 
    \begin{enumerate}
        \item \label{item:u div-free downhill} $\nabla \cdot u =0$,
        \item \label{item:u regularity downhill} for all $n \in \N$,
        \[\|u\|_{L^\infty([0,\infty), W^{n,\infty}(\T^d))} \leq  \frac{C(n) \eta |b|^n (|a+b|^2 - |a|^2)}{\alpha_{a,b} |a|} \sum_{k \geq 0} \frac{|k|^n}{(1+d_{-k})^{n+1}},\]
         \item \label{item:theta fourier support downhill} for all $t \geq 0$, $\supp \hat \theta_t \subseteq \Z^2 - B_{|a+b|}$,
        \item \label{item:theta mass moved downhill} for some $\beta \in \C$, $\theta_{T} = \beta f_{a+b}.$
    \end{enumerate}
\end{corollary}

\begin{proof}
    We note that our assumptions directly allow us to apply Proposition~\ref{prop:flow down hill} to give a coefficient field $v^k_t$ with $v^{-k}_t = \bar v^k_t$ such that for $z^k_t$ the solution to~\eqref{eq:main-ode} with coefficients $d_k, v^k_t$ and initial data $z^k_0 = \delta_{k,0}$, such that for some $\beta \in \C$, 
    \begin{equation}
    \label{eq:downhill mass moved z}
    z^k_{LT} = e^{ALT} \beta \delta_{k,1},
    \end{equation}
    and for all $n \in \N$, there exists $C(n)>0$ such that for all $k \in \N - \{0\}$,
    \begin{equation}
    \label{eq:v bound downhill}
    \sup_{t \geq 0} |v^k_t| \leq C \eta (1 + d_{1-k})^{-n}.
    \end{equation}
    Let $v : [0,\infty) \times \T \to \R$ be given by $\hat v_t(k) = v^K_t$, noting that $v$ is $\R$-valued as $v^{-k}_t = \bar v^k_t$. We then let $u = w_{a,b,v,L}$ as in Definition~\ref{defn:1d-to-any-d}. Then by Proposition~\ref{prop:ode-to-pde}, we have the for $t \geq 0$,
       \begin{equation}
    \label{eq:theta-rep downhill}
    \theta_{t} = e^{-ALt} \sum_{k \in \Z} z^k_{Lt} f_{a+kb}.
    \end{equation}
    We now verify our conclusions. Item~\ref{item:u div-free downhill} is direct from Item~\ref{item:div-free} of Proposition~\ref{prop:ode-to-pde}. Item~\ref{item:theta fourier support downhill} is direct from the representation~\eqref{eq:theta-rep downhill} and the hypothesis that $|a+kb| \geq |a+b|$ for $k \in \Z$. Then Item~\ref{item:theta mass moved downhill} is direct from~\eqref{eq:theta-rep downhill} and~\eqref{eq:downhill mass moved z}. Thus all that remains is to verify Item~\ref{item:u regularity downhill}. To that end, we note that by Item~\ref{item:regularity} of Proposition~\ref{prop:ode-to-pde} and~\eqref{eq:v bound downhill}, 
    \begin{align*}
        \|u\|_{L^\infty_tW^{n,\infty}_x} &\leq \frac{|b|^n L}{\alpha_{a,b}|a|} \|v\|_{L^\infty([0,\infty), W^{n,\infty}(\T^d))}
        \\&\leq \sup_{t \geq 0} \frac{C(n) |b|^n (|a+b|^2 - |a|^2)}{\alpha_{a,b} |a|} \sum_{k \in \Z} |k|^n |v^k_t|
        \\&\leq \frac{C(n) \eta |b|^n (|a+b|^2 - |a|^2)}{\alpha_{a,b} |a|} \sum_{k \geq 1} \frac{|k|^n}{(1+d_{1-k})^{n+1}}
        \\&\leq \frac{C(n) \eta |b|^n (|a+b|^2 - |a|^2)}{\alpha_{a,b} |a|} \sum_{k \geq 0} \frac{|k|^n}{(1+d_{-k})^{n+1}},
    \end{align*}
    as claimed.
\end{proof}

\subsection{Applying the translated results}

With Corollary~\ref{cor:move mass theta} and Corollary~\ref{cor:move mass theta downhill}, we can now prove Propositions~\ref{prop:fourier to fourier 2d}--\ref{prop:fourier to fourier 4d-2} by careful verifying the hypotheses of the corollaries. Proposition~\ref{prop:fourier to fourier 2d} is the most direct.

\begin{proof}[Proof of Proposition~\ref{prop:fourier to fourier 2d}]
    We want to apply Corollary~\ref{cor:move mass theta} with $a = (r,0)$ and $b = (-r,r+1)$. We note $|a \cdot b| = r^2 \leq |r| |r+1| = |a||b|$, as $r \geq 0$. We then need to verify the hypotheses of Corollary~\ref{cor:move mass theta}. With our choice of $a,b$, we get that
    \[d_k = \frac{r^2 (k-1)^2 + (r+1)^2 k^2 - r^2}{(r+1)^2 - r^2} = \frac{r^2 (k-1)^2 + (r+1)^2 k^2 - r^2}{2r + 1}.\]
    Then 
    \begin{equation*}
    M := \min_{k \in \Z, k \ne 0,1} d_k \geq \frac{r^2}{2r+1} \geq \frac{r}{3}\geq \frac{r_0}{3},
    \end{equation*}
    so choosing $r_0$ sufficiently large, we get Item~\ref{item:M large theta} of the hypotheses of Corollary~\ref{cor:move mass theta}. For Item~\ref{item:S small theta}, we have that
    \begin{equation}
    \label{eq:S bound L infty}
    S := \sum_{k \in \Z} \frac{1}{1+d_k} \leq 2 + \frac{2}{r^2} \sum_{k =1}^\infty \frac{1}{ k^2} \leq 6,
    \end{equation}
    using that $r \geq 1.$ For Item~\ref{item:d spacing theta}, we have for $k \in \N - \{0\},$
    \[d_{k+1} - d_{1-k} = \frac{(r+1)^2 ((k+1)^2 - (1-k)^2) }{2r + 1} \geq \frac{4(r+1)^2}{2r+1} \geq \frac{4r}{3} \geq 1.\]
    We then take that velocity field $u$ from Corollary~\ref{cor:move mass theta}. From the conclusions of Corollary~\ref{cor:move mass theta}, we immediately get all the conclusions of Proposition~\ref{prop:fourier to fourier 2d}, except that $\|u\|_{L^\infty_{t,x}} \leq C$, which we now verify. From Corollary~\ref{cor:move mass theta}, we have that
    \[\|u\|_{L^\infty_{t,x}} \leq \frac{C ((r+1)^2 - r^2)}{r \alpha_{a,b}} \Big(1 + e^{-M} \sum_{k=0}^\infty \frac{1}{1+d_{-k}}\Big).\]
    Combining the above with~\eqref{eq:S bound L infty}, we have that
    \[\|u\|_{L^\infty_{t,x}} \leq \frac{C}{\alpha_{a,b}}.\]
    Thus to conclude, we just need to uniformly lower bound $\alpha_{a,b}$. For that, we note that
    \[\alpha_{a,b}^2 = 1 - \frac{r^4}{r^2(r^2+ (r+1)^2)} = 1 -\frac{1}{1 + (1+1/r)^2} = \frac{(1+1/r)^2}{1 + (1+1/r)^2} \geq \frac{1}{2},\]
    allowing us to conclude.
\end{proof}

We now turn our attention to Proposition~\ref{prop:fourier to fourier 3d}. We will start with some $(m,n,\ell) \in \Z^3$ and use Legendre's three square theorem to generate $(x,y,z) \in \Z^3$ that we seek to move to. However, we need to verify that we can choose $(x,y,z)$ in such a way that the hypotheses of Corollary~\ref{cor:move mass theta} are satisfied. We do this by supposing without loss of generality (by appropriately changing coordinates) that $0 \leq m \leq n \leq \ell$ and $(x,y,z) = (\tilde x, - \tilde z, \tilde y)$ where $0 \leq \tilde x \leq \tilde y \leq \tilde z,$ which we can do since we only care about $|(x,y,z)|$. There is certainly a more direct but tedious argument showing that this choice of $(m,n,\ell),(x,y,z)$ verifies the hypotheses of Corollary~\ref{cor:move mass theta} for sufficiently large $|(m,n,\ell)|$, but we choose to use the cleaner soft argument below.

\begin{lemma}
\label{lem:sphere-geometry}
    Let $A,C \subseteq S^{d-1} \subseteq \R^d$ closed with $A = - A, C = -C,$ and $A \cap C = \emptyset$. For $\delta >0$ and a set $R \subseteq \R^d$, let $R_\delta$ denote the $\delta$-thickening:
    \[R_\delta := \{r + v : r \in R, |v| \leq \delta\}.\]
    For $a,b \in \R^{d}$, let
    \[\alpha_{a,b} := \Big(1 - \frac{(a\cdot b)^2}{|a|^2|b|^2}\Big)^{1/2}.\]
    Then there exists $K,\delta>0$ such that
    \begin{enumerate}
        \item \label{item:difference isnt zero} $0 \not \in C_\delta - A_\delta$.
        \item \label{item:angle uniformly bounded} $\inf_{a \in A_\delta, c \in C_\delta} \alpha_{a,c-a} \geq K^{-1}.$
        \item \label{item:right growth} For all $k \in \Z$, if $k \ne 0,1,$ then 
        \[\inf_{a \in A_\delta, c \in C_\delta} |a + k(c-a)| \geq K^{-1}|k| +1.\]
        \item \label{item:right separation} For all $k \in \N-\{0\}$,
        \[\inf_{a \in A_\delta, c \in C_\delta} |a + (k+1)(c-a)|^2 - |a + (1-k)(c-a)|^2 \geq K^{-1} k.\]
    \end{enumerate}
    
\end{lemma}

\begin{proof}
    For Item~\ref{item:difference isnt zero}, we note that since $A \cap C \ne \emptyset$, $0 \not \in C-A$. Then $C_\delta - A_\delta \subseteq (C-A)_{2\delta}$, so taking $\delta$ small enough, we have $0 \not \in (C-A)_{2\delta}$, allowing us to conclude.

    For Item~\ref{item:angle uniformly bounded}, we note that it suffices to prove that for all $a \in A, c\in C$, $\alpha_{a,c-a} \ne 0$, and then continuity and compactness allow us to conclude the general statement. Note that by definition of $\alpha_{a,c-a}, \alpha_{a,c-a} =0$ if and only if $c-a = \gamma a$ for some $\gamma \in \R$. But then $c = (1+\gamma) a$, but since $|c| = |a| = 1$, we must have $\gamma = 0, -2$, so either $c = a$ or $c = -a$. But since $A = -A$, in either case, we get that $C \cap A \ne \emptyset$, contradicting our hypothesis. 

    For Item~\ref{item:right growth}, we note that since $0 \not \in C_\delta - A_\delta$, there exists some $\ep>0$ such that for all $a \in A_\delta, c \in C_\delta$, $|c-a| \geq \ep$. Then we note that
    \[|a + k(c-a)| \geq |k| |c-a| -|a| \geq |k| \ep - 1-\delta.\]
    Thus we get the desired result, choosing $K$ large enough, for all $|k| > 2 \ep^{-1}$, assuming without loss of generality that $\delta <1$. Thus we only need to prove the statement for $|k| \leq 2 \ep^{-1}$, that is a finite collection of $k$. Thus by compactness and continuity, it suffices to prove for all $k \in \Z$, $k \ne 0,1,$ and for all $a\in A, c \in C$ that $|a + k(c-a)| >1$. This then is a consequence of the claim that for all $t \in \R, t \not \in [0,1]$,
    \[(1-t)a + tc \not \in \overline{B_1}.\]
    Suppose that that for some $x \in \overline{B_1}$, $(1-t) a + tc = x$, then if $t >1$,
    \[c = \frac{1}{t} x - \frac{1-t}{t} a,\]
    so $c$ is a convex combination of points $x,a \in \overline{B_1}$. However, $c$ is an extreme point of $\overline{B_1}$ as $c \in S^{d-1}$, so this is a contradiction. The argument follows symmetrically, using that $a$ is an extreme point, in the case that $t<0.$ Thus we conclude Item~\ref{item:right growth}.

    Finally, for Item~\ref{item:right separation}, expanding out the norms with the dot product, we have that
    \begin{align*}&|a + (k+1)(c-a)|^2 - |a + (1-k)(c-a)|^2 
    \\&\qquad= \big((k+1)^2 - (1-k)^2\big) (|c|^2 + |a|^2 -2 c \cdot a) + 2\big((k+1) - (1-k)\big)(a \cdot c -|a|^2)
    \\&\qquad= 4k \big(|c|^2 - c \cdot a\big).
    \end{align*}
    Thus the result follows provided we can show that 
    \[\inf_{a \in A_\delta, c \in C_\delta} |c|^2 - c \cdot a>0.\]
    From compactness and continuity, it suffices to show that for all $a \in A, c \in C$, $|c|^2 - c \cdot a>0$. Then since $a,c \in S^{d-1}$, we note that $|c|^2 - c\cdot a \geq |c|^2 - |c||a| = 0$, with equality if and only if $a = c$ or $a = -c$. However since $A = - A$ and $A \cap C = \emptyset$, this is impossible, allowing us to conclude. 
\end{proof}

We now can apply Lemma~\ref{lem:sphere-geometry} to verify the hypotheses of Corollary~\ref{cor:move mass theta} and conclude Proposition~\ref{prop:fourier to fourier 3d}.

\begin{proof}[Proof of Proposition~\ref{prop:fourier to fourier 3d}]
    Let $(m,n,\ell) \in \Z^3$ with $|(m,n,\ell)| \geq r_0$. Then by Legendre's three square theorem, every number that is $1 \pmod{4}$ is the sum of three squares, thus there exists $(\tilde x,\tilde y,\tilde z) \in \Z^3$ with 
    \[|(m,n,\ell)|^2 +1 \leq |(x,y,z)|^2 \leq |(m,n,\ell)|^2 +8.\]
    Without loss of generality, let's assume that $0 \leq m \leq  n \leq \ell$ and $0 \leq \tilde x \leq \tilde y \leq \tilde z.$ We then define
    \[c:= (x,y,z) := (\tilde x,-\tilde z,\tilde y).\]
    We then define
    \[a := (m,n,\ell) \quad \text{and}\quad b := (x,y,z) -(m,n,\ell).\]
    We want to apply Lemma~\ref{lem:sphere-geometry} in order to verify the hypotheses of Corollary~\ref{cor:move mass theta}. We let 
    \[A := \{(u,v,w) \in S^2 : 0 \leq u \leq v \leq w \text{ or } w \leq v \leq u \leq 0\},\]
    and
    \[C := \{(u,-w,v) \in S^2 : 0 \leq u \leq v \leq w \text{ or } w \leq v \leq u \leq 0\}.\]
    We note that $A = -A, C = -C$, and $A,C$ are closed. We check that $A \cap C =\emptyset$. Suppose that $(u,v,w) \in A \cap C$. Then either $w >0$ or $w < 0$. Let's suppose the first case, the other follows symmetrically. So we have that since $(u,v,w) \in A$ and $w>0$, $0 \leq u \leq v \leq w$. However since $(u,v,w) \in C$ with $w>0$, we must be in the case that $(u,v,w) = (f,-h,g)$ with $0 \leq f \leq g \leq h$. But then $h>0$, so $v<0$, contradicting $0 \leq v$. Thus $A \cap C = \emptyset.$

    Thus we are in the setting of Lemma~\ref{lem:sphere-geometry}, so let $K, \delta>0$ as in Lemma~\ref{lem:sphere-geometry}. We note that 
    \[\frac{a}{|a|} \in A,\quad \frac{c}{|c|} \in C,\quad\text{and}\quad 1 \leq \frac{|c|}{|a|} \leq \sqrt{1 + \frac{8}{|a|^2}}.\]
    Thus $\frac{a}{|a|} \in A \subseteq A_\delta$ and if we take $r_0$ large enough, $\frac{c}{|a|} \in C_\delta.$ Note then that
    \[\frac{b}{|a|} = \frac{c}{|a|} - \frac{a}{|a|}.\]
    Then from Item~\ref{item:angle uniformly bounded}, Item~\ref{item:right growth}, and Item~\ref{item:right separation} of Lemma~\ref{lem:sphere-geometry}, we have that, for some universal $K>0,$
    \begin{enumerate}
        \item \label{item:alpha-bounded-3d} $ \alpha_{a,b} = \alpha_{\frac{a}{|a|}, \frac{b}{|a|}} \geq K^{-1},$
        \item \label{item:growth 3d} for all $k \in \Z, k \ne 0,1,$\; $|a + kb| \geq K^{-1}|a|  |k| + |a|,$
        \item \label{item:separation 3d} for all $k \in \N  - \{0\}$, $|a + (k+1)b|^2 - |a + (1-k)b|^2 \geq K^{-1} |a| k$.
    \end{enumerate}
    We then take as in Corollary~\ref{cor:move mass theta}
    \[d_k := \frac{|a+kb|^2}{L} - A,\quad A := \frac{|a|^2}{L},\quad L:= |a+b|^2- |a|^2,\]
    and
    \[M:= \min_{k\in \Z, k \ne 0,1} d_k \quad\text{and}\quad S:= \sum_{k \in \Z} \frac{1}{1+d_k}.\]
    We first note that $1 \leq L \leq 8$.  We now see that Item~\ref{item:growth 3d} and Item~\ref{item:separation 3d} above above readily imply Item~\ref{item:M large theta}, Item~\ref{item:S small theta}, and Item~\ref{item:d spacing theta} of Corollary~\ref{cor:move mass theta} using that $|a| \geq r_0$ and provided we take $r_0$ sufficiently large.
    
    We then take the velocity field $u$ from Corollary~\ref{cor:move mass theta}, which immediately gives the conclusions of Proposition~\ref{prop:fourier to fourier 3d}, except that $\|u\|_{L^\infty_t W^{1,\infty}_x} \leq C,$ which we now verify. By Item~\ref{item:u regularity} of the conclusions of Corollary~\ref{cor:move mass theta}, the lower bound on $\alpha_{a,b}$ in Item~\ref{item:alpha-bounded-3d} above, and that $L \leq 8,$ we see that
    \[\|u\|_{L^\infty_t W^{1,\infty}_x} \leq C \frac{|b|}{|a|} \Big(1 + e^{-M} \sum_{k=0}^\infty \frac{|k|}{(1 + d_{-k})^2}\Big) \leq C,\]
    using the decay of $d_{-k}$ given by Item~\ref{item:growth 3d} and that $|b| \leq |a| + |c| \leq 3|a|.$ Thus we conclude.
\end{proof}

Proposition~\ref{prop:fourier to fourier 4d-1} follows fairly directly from Corollary~\ref{cor:move mass theta}.

\begin{proof}[Proof of Proposition~\ref{prop:fourier to fourier 4d-1}]
    We let $a = (m,n,\ell,p)$ and $b = (0,0,0,-2p-1).$ Then as in Corollary~\ref{cor:move mass theta}, we define    \[d_k := \frac{|a+kb|^2}{L} - A,\quad A := \frac{|a|^2}{L},\quad L:= |a+b|^2- |a|^2.\]
    We note that
    \[L = 2p+1 \quad \text{and} \quad d_k = \frac{((1+p)k - p)^2 -p^2}{2p+1}.\]
    Thus for $k \ne 0,1$, we have that
    \begin{equation}
    \label{eq:dk decay 4d1}
    d_k \geq C^{-1} p |k|^2.
    \end{equation}
    Choosing $p \geq 10$ sufficiently large, we get Item~\ref{item:M large theta} and Item~\ref{item:S small theta} of the hypotheses of Corollary~\ref{cor:move mass theta}. For Item~\ref{item:d spacing theta}, we have for $k \in \N - \{0\},$
    \begin{align*}d_{k+1} - d_{1-k} &= \frac{((1+p) (k+1) - p)^2  -((1+p)(1-k) -p)^2}{2p+1}
    \\&= \frac{4(1+p)^2k + 4(1+p)pk}{2p+1} \geq p.
    \end{align*}
    Taking $p$ large enough, we thus get Item~\ref{item:d spacing theta} of the hypotheses of Corollary~\ref{cor:move mass theta}. 

    We then take the velocity field $u$ from Corollary~\ref{cor:move mass theta}, which immediately gives the conclusions of Proposition~\ref{prop:fourier to fourier 4d-1}, except that $\|u\|_{L^\infty_t W^{s,\infty}_x} \leq C(s)$, which we now verify. By Item~\ref{item:u regularity} of the conclusions of Corollary~\ref{cor:move mass theta}, using the decay of $d_k$ given by~\eqref{eq:dk decay 4d1}, we have that for each $s\in \N$,
    \[\|u\|_{L^\infty_t W^{s,\infty}_x} \leq \frac{C(s)}{\alpha_{a,b}},\]
    where $C(s)$ can also freely depend on $p$, which is fixed. Thus to conclude, we just need lower bound $\alpha_{a,b}$. Note that, using that $(m,n,\ell) \ne 0,$
    \[\alpha_{a,b}^2 = 1 - \frac{p^2 (2p+1)^2}{(2p+1)^2 (p^2 + m^2 + n^2 + \ell^2)} \geq  1 - \frac{p^2 (2p+1)^2}{(2p+1)^2 (p^2 + 1)} \geq C^{-1} >0,\]
    allowing us to conclude.
\end{proof}

Finally, we use Corollary~\ref{cor:move mass theta downhill} to conclude Proposition~\ref{prop:fourier to fourier 4d-2}.

\begin{proof}[Proof of Proposition~\ref{prop:fourier to fourier 4d-2}]
    Let $(m,n,\ell) \in \Z^3$ with $|(m,n,\ell)| \geq r_0$ and we assume without loss of generality that $0 \leq m \leq n \leq \ell$. We then choose $(x,y,z) \in \Z^3$ as in the proof of Proposition~\ref{prop:fourier to fourier 3d}. We then let
    \[a := (m,n,\ell,-p-1),\quad c := (x,y,z,p),\quad \text{and} \quad b := c-a.\]
    We define, as in Corollary~\ref{cor:move mass theta downhill}, 
     \[d_k := \frac{|a+kb|^2}{L} - A,\quad A:= \frac{|a+b|^2}{L}, \quad L := |a|^2 - |a+b|^2.\]
     We note that $1 \leq L \leq 3p$, using that $(p+1)^2 - p^2 \geq 2p + 1 \geq 20$ and $|(x,y,z)|^2 \leq |(m,n,\ell)|^2 + 8$. We note that 
     \begin{equation}
     \label{eq:dk bound p downhill}
     d_k \geq \frac{|kp -(p+1)|^2 - p^2 -8}{3p} \geq C^{-1} |k|^2 p,
     \end{equation}
     for $k \ne 0,1$. Thus we have Item~\ref{item:M large theta downhill} and Item~\ref{item:S small theta downhill} of the hypotheses of Corollary~\ref{cor:move mass theta downhill} using the $p$ is large enough. Item~\ref{item:d spacing theta downhill} of the hypotheses of Corollary~\ref{cor:move mass theta downhill} is verified by combining the arguments for the analogous item in the proof of Proposition~\ref{prop:fourier to fourier 3d} and the proof of Proposition~\ref{prop:fourier to fourier 4d-1}.

     Thus we can apply Corollary~\ref{cor:move mass theta downhill} for some $\eta>0$ to be determined to give the velocity field $u$, which immediately gives the conclusions of Proposition~\ref{prop:fourier to fourier 4d-2} except that $\|u\|_{L^\infty_t W^{s,\infty}_x} \leq C(s)$, which we now verify. By Item~\ref{item:u regularity downhill} of Corollary~\ref{cor:move mass theta downhill}, we have that
     \[\|u\|_{L^\infty_tW^{s,\infty_x}} \leq \frac{C(s) \eta |b|^s (|a+b|^2 - |a|^2)}{\alpha_{a,b} |a|} \sum_{k \geq 0} \frac{|k|^s}{(1+d_{-k})^{s+1}} \leq C(s)\eta \frac{|b|^s}{\alpha_{a,b} |a|},\]
     using~\eqref{eq:dk bound p downhill} to lower bound $d_{-k}$ and letting $C(s)$ depend on $p$, which is fixed. We then bound $|b| \leq 2 |a| + C$, giving
     \[\|u\|_{L^\infty_tW^{s,\infty}_x} \leq C(s) \eta (|a|^s + 1) \alpha_{a,b}^{-1}.\]
     Noting that $\alpha_{a,b} \ne 0$, which e.g.\ follows from the lower bound on $\alpha_{a,b}$ in the proof of Proposition~\ref{prop:fourier to fourier 3d}, recalling that we can choose $\eta$ depending on $a,b$, we then take $\eta = e^{-|a|} \alpha_{a,b}$, giving that 
          \[\|u\|_{L^\infty_tW^{s,\infty}_x} \leq C(s) e^{-|a|} (|a|^s + 1) \leq C(s),\]
    allowing us to conclude.
\end{proof}

{\small
\bibliographystyle{alpha}
\bibliography{keefer-references}
}

\end{document}